\documentclass[11pt]{amsart}
\usepackage{amsfonts, amsmath, amssymb, amsthm, amscd, color, bookmark,graphicx}
\usepackage[latin1]{inputenc}

\usepackage{hyperref}
\hypersetup{
  pdfborder={0 0 0},
  colorlinks   = true, 
  urlcolor     = blue, 
  linkcolor    = blue, 
  citecolor   = red 
}

\newtheorem{thm}{Theorem}[section]
\newtheorem{prop}[thm]{Proposition}
\newtheorem{lemma}[thm]{Lemma}
\newtheorem{cor}[thm]{Corollary}

\theoremstyle{definition}
\newtheorem{defn}[thm]{Definition}
\newtheorem{question}[thm]{Question}

\theoremstyle{remark}
\newtheorem{ex}[thm]{Example}
\newtheorem{rem}[thm]{Remark}
\newtheorem{exercise}[thm]{Exercise}

\newcommand{\C}{\mathrm{C}}
\newcommand{\St}{\mathrm{St}}

\newcommand{\No}{{\rm N}}

\newcommand{\GG}{\Gamma\mathfrak{G}}

\DeclareMathOperator{\link}{link}

\DeclareMathOperator{\rk}{rk}
\DeclareMathOperator{\cd}{cd}
\DeclareMathOperator{\vb}{vb_1}
\DeclareMathOperator{\bone}{b_1}

\newcommand{\cX}{\mathcal{X}}

\newcommand{\cT}{\mathcal{T}}

\newcommand{\N }{\mathbb{N}}
\newcommand{\Z }{\mathbb{Z}}
\newcommand{\Q }{\mathbb{Q}}

\newcommand{\vr}{ \leqslant_{vr}}
\newcommand{\leqs}{\leqslant}
\newcommand{\VR}{(LR)}
\newcommand{\VRC}{(VRC)}

\begin{document}

\title{Virtual retraction properties in groups}

\author{Ashot Minasyan}
\address{School of Mathematical Sciences,
University of Southampton, Highfield, Southampton, SO17 1BJ, United
Kingdom.}
\email{aminasyan@gmail.com}
%

\begin{abstract}
If $G$ is a group, a virtual retract  of $G$ is a subgroup which is a retract of a finite index subgroup. Most of the paper focuses
on two group properties: property \VR{}, that all finitely generated subgroups are virtual retracts, and property \VRC{}, that all cyclic subgroups are virtual retracts.
We study the permanence of these properties under commensurability, amalgams over retracts, graph products and wreath products.
In particular, we show that \VRC{} is stable under passing to finite index overgroups, while \VR{} is not.

The question whether all finitely generated virtually free groups satisfy \VR{} motivates the remaining part of the paper, studying virtual
free factors of such groups. We give a simple criterion characterizing when a finitely generated subgroup of a virtually free group is a free factor
of a finite index subgroup. We apply this criterion to settle a conjecture of Brunner and Burns.
\end{abstract}


\maketitle

\section{Introduction}
Recall that a subgroup $H$ of a group $K$ is called a \emph{retract} if there is a homomorphism $\rho:K \to H$ which restricts to the identity map on $H$.  This is equivalent to $K$ splitting as a semidirect product $N \rtimes H$, where $N=\ker\rho$. In this case the map
$\rho$ is called a \emph{retraction} of $K$ onto $H$.

\begin{defn}\label{def:vr} Let $G$ be a group and let $H$ be a subgroup of $G$. We will say that $H$ is a \emph{virtual retract of $G$},
denoted $H \vr G$, if there exists a subgroup
$K \leqslant G$ such that $|G:K|<\infty$, $H \subseteq K$ and $H$ is a retract of $K$.
\end{defn}

Note that, according to this definition, any finite index subgroup of $G$ is a virtual retract.

\begin{defn}\label{def:VR_and_VRC} Let $G$ be a group. We will say that $G$ has \emph{property \VRC{}} if every cyclic subgroup of $G$ is a virtual retract. If
all finitely generated subgroups of $G$ are virtual retracts then we will say that $G$ has \emph{property \VR{}}.
\end{defn}

Property \VR{} is much stronger than \VRC{}; for example, the direct product of two non-abelian free groups has \VRC{} but does not have \VR{}
(cf. Lemma~\ref{lem:vrab-basics}.(b) and Remark~\ref{rem:VR_dir_prod} below).
Explicitly, both of these properties were first introduced by Long and Reid in \cite{L-R}, however, implicitly they were investigated much earlier.
One of the purposes of this article is to emphasize that properties \VR{}, \VRC{}, and virtual retracts in general, apart from having numerous applications
are also very interesting by themselves.

Virtual retractions are extremely useful for studying the profinite topology on groups. It is well-known that a virtual retract of a residually finite group is closed in the
profinite topology (see Lemma~\ref{lem:virt_retr->closed}), hence property \VRC{} implies that the group is cyclic subgroup separable, and property \VR{}
yields that the group is LERF. In particular, Scott \cite{Scott} proved that all surface groups are LERF essentially by showing that they satisfy property \VR{}.
This was greatly generalized by Wilton \cite{Wilton-limit}, who showed that all limit groups (a.k.a. finitely generated fully residually free groups) satisfy \VR{}.

Knowing that a subgroup is a virtual retract gives a lot of insight into its embedding in the ambient group.
For example, if $G$ is a finitely generated group and $H \vr G$ then $H$ is  finitely generated and undistorted (quasi-isometrically embedded) in $G$.
Many other finiteness properties of $G$, including finite presentability, would also be inherited by $H$.
And conversely, in some classes of groups all ``nicely embedded'' subgroups are virtual retracts:  Davis \cite{Davis} proved this for undistorted
subgroups of finitely generated free nilpotent groups, Long and Reid \cite{L-R} showed this for geometrically finite subgroups of  right angled hyperbolic reflection groups,
and Haglund \cite{Haglund} established this for strongly quasiconvex subgroups of right angled Coxeter and Artin groups.

Another property of profinite topology, where virtual retractions  play an important role, is conjugacy separability. A group is said to be \emph{conjugacy separable}
if given two non-conjugate elements, there is a homomorphism to a finite group such that the images of these elements are still non-conjugate.
Conjugacy separability is much more sensitive than residual finiteness. In particular, it is not stable under commensurability: a conjugacy separable group could have an index $2$ subgroup or overgroup which is not conjugacy separable (see \cite{Min-subdir} and references therein). However, a retract of a conjugacy separable group is also conjugacy separable. By an earlier work of the author \cite{M-RAAG},
all finite index subgroups of right angled Artin groups are conjugacy separable. On the other hand, in the seminal paper \cite{H-W} Haglund and Wise initiated a program
aiming to prove that many groups arising in Geometric Group Theory have finite index subgroups that are virtual retracts of right angled Artin groups
(see, for example,  \cite{Wise-QH} and \cite{Agol} for successful implementations of this program).
Thus any such ``virtually compact special group'' $A$ (using the terminology of \cite{H-W}) contains a conjugacy separable subgroup $H$ of finite index. Unfortunately
this does not guarantee that $A$ is itself conjugacy separable. An attempt to prove conjugacy
separability of $A$ naturally prompts the following question.

\begin{question}\label{q:quest-1} Suppose that $G$ is a group and $H \vr G$. Find conditions on $G$ and $H$ ensuring that for every subgroup
$A \leqslant G$, which contains $H$ as a finite index subgroup, one has $A \vr G$.
\end{question}

This question turns out to be quite interesting. If a group $X$ is not virtually abelian then for $G=X \wr \Z_2 \cong (X \times X)\rtimes \langle \alpha\rangle_2$,
the diagonal subgroup $H \leqs X \times X$ is a virtual retract, but its index $2$ overgroup $A=\langle H,\alpha \rangle \cong X \times \Z_2$ is not: see Example~\ref{ex:wreath} and Proposition~\ref{prop:wreath_ex}. In this case $H \cong X$ is not virtually abelian.
In Section \ref{sec:retr_on_vab} we establish a strong converse to this example.

\begin{thm}\label{thm:virt_ab-lift} Let $G$ be a residually finite group, let $H \vr G$ be finitely generated and virtually abelian, and
let $A \leqslant G$ be any subgroup containing $H$ with finite index. Then $A \vr G$.
\end{thm}

This theorem implies that property \VRC{} is stable under commensurability, making it much more amenable to study.
Another application is the following proposition, proved in
Section~\ref{sec:VR_and_VRC}.

\begin{prop} \label{prop:VRC->VRAb} If $G$ is a group with property \VRC{} then every finitely generated virtually abelian subgroup is a virtual retract of $G$.
\end{prop}

The claim of Proposition \ref{prop:VRC->VRAb} is somewhat surprising, as it shows that property \VRC{} is actually stronger than one would originally think.
In Section \ref{sec:amalg_over_retr} we show that \VRC{} is preserved by amalgamated products over retracts and by HNN-extensions/amalgamated products
over finite subgroups (see Theorem~\ref{thm:vrc-stab_apr}
and Corollary~\ref{cor:VRC_stable_under_amalg_over_finite}). This easily implies that graph products of groups with \VRC{} also satisfy \VRC{}
(Theorem~\ref{thm:gp_of_vrc}).

\begin{cor}\label{cor:virt_spec->vrc} Any group $G$ possessing a finite index subgroup that embeds in a right angled Artin or Coxeter group has property \VRC{}.
\end{cor}

The above corollary covers all ``virtually special''  groups of Haglund and Wise \cite{H-W}. Combined with Proposition~\ref{prop:VRC->VRAb} it implies that
any virtually abelian subgroup of such a group is a virtual retract. The fact that right angled Artin/Coxeter groups virtually have \VRC{}, has already been
observed by Aschenbrenner,  Friedl and Wilton \cite[Chapter 6, (G.18)]{A-F-W}, based on an earlier result of Agol \cite{Agol-RFRS}.

In Section \ref{sec:quasipot} we discuss an application of \VRC{} to quasi-potency, which is a condition originally introduced by Evans \cite{Evans}
and later rediscovered by Tang \cite{Tang} and Burillo-Martino \cite{Bur-Mar}.
This condition provides good control over the profinite topology on an amalgamated product of two groups over a (virtually) cyclic subgroup.
We use it to show that the amalgamated product of two virtually special groups
over a virtually cyclic subgroup is cyclic subgroup separable (Corollary~\ref{cor:virt_spec_over_virt_cyclic}).
As far as the author is aware, even the residual finiteness of such amalgamated products was previously unknown in general (of course, in some cases
such groups are themselves virtually special, so much more is true: see \cite{H-W-amalg,Wise-QH}).

Section \ref{sect:solv} is devoted to solvable groups and wreath products. We observe that a virtually polycyclic group has \VRC{} if and only if it is virtually abelian (Proposition~\ref{prop:vpolyc+vrc->vab}). On the other hand, in Theorem~\ref{thm:wreath->VRAb} we prove that the restricted wreath product of two groups with \VRC{}
has \VRC{} if and only if either the base group is abelian or the acting group is finite. Theorem~\ref{thm:vr_for_wreath_with_Z} shows that for a finitely generated abelian group
$A$, the wreath product $A \wr \Z$  satisfies property \VR{} if and only if $A$ is semisimple. This leads to the following example.

\begin{prop}\label{prop:f_i_overgp-not_VR} The group $G={\Z_2}^2 \, \wr\, \Z$ satisfies property \VR{}, but has an index $2$ overgroup $\tilde G$ which does not satisfy \VR.
\end{prop}

Thus \VR{}, unlike \VRC{}, is not invariant under commensurability. This property is also not stable under direct products or amalgamated free products over retracts
(see Remark~\ref{rem:VR_dir_prod}.(b) and Example~\ref{ex:VR_not_closed_under_APRs}). In \cite{G-M-S} Gitik, Margolis and Steinberg proved that \VR{} is preserved
by free products, but extending this to amalgamated free products/HNN-extensions of groups with \VR{} over finite subgroups does not seem to be easy.
In fact, we do not even know whether all finitely
generated virtually free groups satisfy \VR{} (see Question~\ref{q:v_free->VR}). Motivated by this question, in Section~\ref{sec:vf} we investigate when a
finitely generated
subgroup of a virtually free group is a ``virtual free factor''.

A group is said to have \emph{M. Hall's property} if every finitely generated subgroup is a free factor of a subgroup of finite index.
Evidently this is much stronger than \VR{}; the name comes from the
fact that this property was originally proved to hold in free groups by M. Hall \cite{M-Hall} (see also \cite{Burns1} for an explicit statement).

In \cite{Brun-Burns} Brunner and Burns studied groups satisfying M. Hall's property. Their results, combined with Dunwoody's Accessibility Theorem \cite{Dunw},
imply that a finitely presented group with this property must be virtually free.
In \cite{Bogop-1,Bogop-2,Bog-habil} Bogopol'ski\v{\i} used the theory of covering spaces to study this property for fundamental groups of finite graphs of finite groups.
In particular, in \cite{Bog-habil} he showed that a finitely generated group $G$ satisfies M. Hall's property if and only if it is virtually free and every finite subgroup of $G$ is a free factor of a finite index subgroup.

Bogopol'ski\v{\i}'s work was motivated by a conjecture of Brunner and Burns \cite{Brun-Burns}, stating that a finitely generated
virtually free group $G$ has M. Hall's property if and only if $|\No_G(F):F|<\infty$ for every non-trivial finitely generated
(equivalently, finite) subgroup $F \leqs G$. The next statement is a simplification of Theorem \ref{thm:Hall_for_vf} from Section \ref{sec:vf}.

\begin{thm}\label{thm:Hall-crit-simple} Let $F$ be a finitely generated subgroup of a finitely generated virtually free group $G$. Then $F$ is a free factor of a finite index subgroup of $G$ if
and only if  $|\C_G(f):\C_F(f)|<\infty$ for every non-trivial finite order element $f \in F$.
\end{thm}

The criterion from Theorem \ref{thm:Hall-crit-simple} provides a positive solution to the above conjecture of Brunner and Burns from \cite{Brun-Burns}.

\begin{cor}\label{cor:Brun-Burns} A finitely generated virtually free group $G$ satisfies M. Hall's property if and only if every non-trivial finitely generated subgroup of
$G$ has finite index in its normalizer.
\end{cor}

Section \ref{sec:open_q} collects several open questions which naturally arose in the course of writing this paper. The Appendix at the end of the paper
investigates a counter-example to the conjecture of Burns and Brunner proposed in \cite{Bogop-1}.

\noindent {\bf Acknowledgements} The author would like to thank Peter Kropholler, Ian Leary, Armando Martino, Nansen Petrosyan and Pavel Zalesskii for
valuable discussions. He is also indebted to Ian Leary for coming up with the graph of groups covers used in the Appendix. Finally, the author thanks the anonymous referee for his/her comments leading to improvements of the exposition.

\section{Background}
\subsection{Notation}
Throughout the paper $\N=\{1,2,3,\dots\}$ will denote the set of {natural numbers}. As usual,  $\Z$ will denote the integers and $\Q$ -- the rational
numbers. Given any $m \in \N$, we will write $\Z_m$ for the group of residues modulo $m$ (isomorphic to the cyclic group of order $m$).

If $G$ is a group and $g,h \in G$, the \emph{commutator} $[g,h]$ will be written as $ghg^{-1}h^{-1}$.
For an element $f \in G$ and a subgroup $H \leqs G$, $\C_G(f)$ and $\C_G(H)$ will denote the corresponding centralizers in $G$, and
$\No_G(H)$ will denote the normalizer of $H$ in $G$. If $A,B$ are two subsets of $G$, $AB=\{ab \mid a \in A, b \in B\} \subseteq G$ will denote their product.
The following fact will be used frequently in the paper:
if $A$ and $B$ are subgroups, one of which normalizes the other, then $AB=BA$ is also a subgroup of $G$.

All actions considered in the paper will be left actions.

\subsection{Trees}
For our purposes here, a \emph{tree} is a graph in the sense of Serre \cite[Chapter I, \S 2]{Serre}, which is non-empty, connected and does not contain any cycles.
Group actions on trees will all be simplicial (i.e., vertices must be mapped to vertices and edges must be mapped to edges) and without edge inversions.
We will say that  an action of a group $G$ on a tree $\cT$ is \emph{cocompact} if the quotient $G\backslash \cT$ is a finite graph.
Trees come equipped with the natural edge-path metric. We shall think of each edge $e$ being oriented: it will have an initial vertex $e_-$ and a terminal vertex $e_+$. The action of a group on a tree will always be assumed to preserve the orientation of the edges.

Let $G$ be a group acting on a tree $\cT$. Suppose that $g \in G$ and $e$ is an edge of $\cT$, and set $e'=g\,e$. We will say that \emph{$g$ translates the edge $e$} if
$e \neq e'$ and the vertices $e_+$ and $e'_-$ belong to the geodesic segment $[e_-,e'_+]$ in $\cT$.
An element $g \in G$ is \emph{hyperbolic} if it translates at least one edge of $\cT$. In this case $g$ has infinite order and
the set of all edges translated by $g$ in $\cT$ forms a simplicial line (homeomorphic to $\mathbb R$), called the \emph{axis of $g$},
on which $g$ acts by translation (see \cite[Chapter I, Propositions 4.11, 4.13]{D-D}).

It is a well-known fact (cf.  \cite[Chapter I, Proposition 4.11]{D-D}) that every element $g \in G$ is either hyperbolic
or \emph{elliptic}. The latter means that $g$ fixes at least one vertex of $\cT$.

\begin{lemma}[{\cite[Chapter I, Theorem 4.12 and Proposition 4.13]{D-D}}] \label{lem:min_inv_subtree}
Let $G$ be a finitely generated group acting on a tree $\cT$. If $G$ has no hyperbolic
elements then it fixes a vertex of $\cT$. Otherwise  the set of all edges translated by elements of $G$ forms the unique minimal $G$-invariant subtree $\cX$ of $\cT$,
and $G$ acts on $\cX$ cocompactly.
\end{lemma}

In particular, this lemma implies that each hyperbolic element $g \in G$ has a unique axis, so the action of the normalizer
$\mathrm{N}_G(\langle g \rangle)$ on $\cT$ must preserve this axis.

\subsection{Profinite topology}
Let $G$ be a group. The left cosets to finite index subgroups of $G$ form a basis of the \emph{profinite topology on $G$}. It is easy to see that $G$ is residually finite if and only if this topology is Hausdorff.

A subgroup of $G$ is closed in the profinite topology if and only if it is equal to the intersection of finite index subgroups.
The group $G$ is said to be \emph{cyclic subgroup separable} if all cyclic subgroups of $G$ are closed in the profinite topology.
A much stronger property is subgroup separability: $G$ is \emph{LERF} (a.k.a. \emph{subgroup separable}) if every finitely generated subgroup of $G$ is closed.
Both of these properties imply residual finiteness.

\begin{lemma}\label{lem:virt_retr->closed} If $G$ is a residually finite group and $H \vr G$ then $H$ is closed in the profinite topology on $G$.
\end{lemma}

\begin{proof} Let $K \leqslant G$ be a finite index subgroup retracting onto $H$. Then $K$ is residually finite, as a subgroup of $G$, so $H$ is closed in
the profinite topology on $K$ by \cite[Lemma 3.9]{Hsu-Wise}. Since $|G:K|<\infty$, every closed subset of $K$ is also closed in $G$, hence $H$ is closed in the
profinite topology on $G$.
\end{proof}

\begin{lemma}\label{lem:vrc->rf} 
If $G$ is a group with property \VRC{} then $G$ is residually finite.
\end{lemma}

\begin{proof}
Indeed, let $g \in G \setminus \{1\}$. Then for some finite index subgroup $K \leqslant G$, containing  $g$, there is a retraction $\rho:K \to \langle g \rangle$.
If $g$ has finite order in $G$, then $\ker\rho$ is a subgroup of finite index of $K$ (and, hence, of $G$) which does not contain $g$. If $g$ has infinite order, there is an epimorphism $\eta: \langle g \rangle \to \Z_2$, where $\Z_2$ is the cyclic group of order two, generated by $\eta(g)$. In this case we see that
$g \notin \ker(\eta\circ \rho)$, and this kernel has finite index in $G$.

Thus for each $g \in G\setminus\{1\}$ we found a finite index subgroup of $G$ which does not contain $g$. This means that $G$ is residually finite.
\end{proof}

\begin{lemma} \label{lem:closed->sep} Let $G$ be a group and $H \subseteq A$ be two subgroups such that $|A:H|<\infty$ and $H$ is closed in the profinite topology on $G$.
Then there exists a finite index subgroup $L \leqslant G$ such that $A \cap L=H$.
\end{lemma}

\begin{proof} Suppose that $A=H \sqcup \bigsqcup_{i=1}^n a_i H$, for some $a_1,\dots,a_n \in A\setminus H$. Since $H$ is closed in the profinite topology on $G$,
there exists a finite index normal subgroup $N \lhd G$ such that $a_i \notin HN$ for each $i=1,\dots, n$. It is easy to see that $L=HN$ is a finite index subgroup of $G$
satisfying the desired property.
\end{proof}

\section{Virtual retractions}
\subsection{Basic properties of virtual retractions.}
Suppose that $H \vr G$, and $K \leqslant G$ is a finite index subgroup containing $H$ and admitting a retraction $\rho:K \to H$. Then $N=\ker\rho \lhd K$, $HN=K$
and $H\cap N=\{1\}$ (in other words, $K \cong N \rtimes H$). Conversely, it is not difficult to see that the existence of $N$ with the above properties is sufficient to show that $H \vr G$.
This observation will be used throughout the paper without further reference.

\begin{rem}\label{rem:equiv_virt_retr} Given a group $G$ with a subgroup $H$, $H$ is a virtual retract of $G$ if and only if there exists a subgroup $N \leqslant G$,
normalized by $H$, such that $|G:HN|<\infty$ and $H \cap N=\{1\}$.
\end{rem}

Let us now state a few basic properties of the relation $\vr$ from Definition \ref{def:vr}.
\begin{lemma}\label{lem:basic_props_vr}
Suppose that $G$ and $G'$ are groups.

\begin{itemize}
  \item[(i)] Let $H \vr G$ and let $A \leqslant G$ be any subgroup containing $H$. Then $H \vr A$.
  \item[(ii)] Suppose that $H\leqslant G$ and $\varphi:G \to G'$ is a homomorphism whose restriction to $H$ is injective and $\varphi(H) \vr G'$. Then $H \vr G$.
  \item[(iii)] If $H \vr G$ and $\alpha:G \to G$ is an automorphism then $\alpha(H) \vr G$. In particular, $gHg^{-1} \vr G$ for every $g \in G$.
   \item[(iv)] If $H \vr G$ and $F \vr H$ then $F \vr G$. In particular, if $H \vr G$ and $F \leqslant H$ is a subgroup with $|H:F|<\infty$ then $F \vr G$.
  \item[(v)] If $H \vr G$ and $H' \vr G'$ then $H \times H' \vr G \times G'$.
\end{itemize}
\end{lemma}

\begin{proof} Claim (i) is obvious (one can just restrict the retraction from a finite index subgroup of $G$ to a finite index subgroup of $A$) and claim (iii) is an immediate consequence of (ii).

To establish claim (ii) it is convenient to assume that $\varphi$ is surjective (which we can do in view of claim (i)) and to use Remark \ref{rem:equiv_virt_retr}.
Thus, if $\varphi(H) \vr G'$ there must exist $N' \leqslant G'$, normalized by $\varphi(H)$,
such that the subgroup $\varphi(H) N'$ has finite index in $G'$ and $\varphi(H) \cap N'=\{1\}$.
Let $N=\varphi^{-1}(N')$ be the full preimage of $N'$ in $G$.
Then $N$ will be normalized by $H$ and $|G:HN|=|G':\varphi(H)N'|<\infty$. Finally, $H \cap N=\{1\}$ in $G$ because $\varphi(H) \cap \varphi(N)=\{1\}$ and $H \cap \ker\varphi=\{1\}$ as $\varphi$ is injective on $H$. Hence $H \vr G$ by Remark \ref{rem:equiv_virt_retr}.

To prove claim (iv), suppose that $H \vr G$ and $F \vr H$. Then there exists a finite index subgroup $K \leqslant G$ and a homomorphism $\rho: K \to H$
such that $H \subseteq K$ and $\rho(h)=h$ for all $h \in H$. Since $F \subseteq H$ we can deduce that $\rho$ is injective on $F$ and $\rho(F)=F \vr H$. Therefore $F \vr K$ by claim (ii),
which implies that $F \vr G$ as $|G:K|<\infty$.

It remains to establish the validity of claim (v). Assuming $H \vr G$ and $H' \vr G'$, we can find finite index subgroups $K \leqslant G$ and $K' \leqslant G'$ such that
$H \subseteq K$, $H' \subseteq K'$, and there are retractions $\rho: K \to H$, $\rho':K' \to H'$. Evidently $\varphi=(\rho,\rho'):K \times K' \to H \times H'$
is a retraction from $K \times K'$, which is a subgroup of finite index in $G \times G'$, onto $H\times H'$. Hence $H \times H' \vr G \times G'$.
\end{proof}

Let $G$ be a group generated by a finite set $S$, and let $H \leqslant G$ be a subgroup generated by a finite set $T$. For $g \in G$ and $h \in H$
we will use $|g|_S$ and $|h|_T$ to denote the \emph{word lengths} of $g$ and $h$ corresponding to the generating sets $S$ and $T$. Recall that the subgroup $H$
is said to be \emph{undistorted in $G$} if there is $C >0$ such that for all $h \in H$ we have $|h|_T \le C |h|_S$. This is equivalent to saying that the
embedding $H \hookrightarrow G$ is quasi-isometric with respect to the word metrics on $H$ and $G$.
Moreover, this notion is independent of the choices of finite generating sets $S$ for $G$ and $T$ for $H$.
The following observation (cf. \cite[Lemma 2.2]{Dav-Ols}) can be used to show that a subgroup is not a virtual retract of a group $G$.

\begin{rem}\label{rem:vr->undist}
Suppose that $G$ is a finitely generated group and $H \leqslant G$. If $H\vr G$ then $H$ is finitely generated and undistorted in $G$.
\end{rem}

\subsection{Promoting retractions to finite index overgroups}
Let us discuss some background to Question \ref{q:quest-1} from the Introduction.
First we observe that in the case when $H$ is actually a retract of $G$, it is sufficient to require residual finiteness of $G$.

\begin{lemma}\label{lem:retract_fin-ext}
Let $G$ be a residually finite group. Suppose that $H$ is a retract of $G$ and $A \leqslant G$ is a subgroup satisfying $H \subseteq A$ and $|A:H|<\infty$.
Then $A \vr G$. In particular, $B \vr G$ for every finite subgroup $B$ of $G$.
\end{lemma}

\begin{proof} Let $\rho:G \to H$ be a retraction and let $N=\ker\rho \lhd G$.
Then $G=HN$ and $H \cap N=\{1\}$; therefore the intersection $A \cap N$ must be finite, as $|A:H|<\infty$. Since $G$ is residually finite, there exists a finite index normal
subgroup $L \lhd G$ such that $L \cap (A \cap N)=\{1\}$. Denote $N'=L\cap N$ and observe that $N' \lhd G$ and $|N:N'|<\infty$. It follows that $HN'$ has finite index in $HN=G$,
so $|G:AN'|<\infty$. Finally, $A \cap N'=A \cap (L \cap N)=\{1\}$ by the choice of $L$, hence $A \vr G$ by Remark \ref{rem:equiv_virt_retr}.

The second claim follows from the first one by applying it to the case when $H=\{1\}$.
\end{proof}

\begin{ex}\label{ex:rf_neces} Let us demonstrate that the assumption of residual finiteness of $G$  is important in Lemma \ref{lem:retract_fin-ext}. Indeed, suppose that $G$ is a group with a non-trivial finite subgroup $A \leqslant G$ such that $A$ is contained in every finite index subgroup of $G$
(for example, this will satisfied if $G$ is infinite and simple). Obviously the trivial subgroup $H=\{1\}$ is a retract of $G$,
but $A$ cannot be a virtual retract of $G$ (even though $|A:H|=|A|<\infty$). Indeed, otherwise there would exist $N \leqslant G$ such that $N$ is normalized by $A$,
$|G:AN|<\infty$ and $A \cap N=\{1\}$. It follows that $|G:N|<\infty$, so $N$ is a finite index subgroup of $G$ which intersects $A$ trivially, contradicting our assumption.
\end{ex}

We will now construct examples showing that in Lemma \ref{lem:retract_fin-ext} one cannot replace the assumption that
$H$ is a retract of $G$ with the weaker assumption that $H\vr G$.

\begin{ex}\label{ex:wreath}
Let $X$ be any group and let $G=X \wr \langle \alpha \rangle_2 \cong X \wr \Z_2$ be the wreath product of
$X$ with a cyclic group of order $2$ generated by $\alpha$. In other words,
$G=(X \times X) \rtimes \langle \alpha \rangle_2$, where $\alpha$ acts on $X \times X$ by interchanging the two factors: $\alpha (x,y) \alpha^{-1}=(y,x)$ for all
$(x,y) \in X\times X$. Observe that the diagonal subgroup $H=\{(x,x) \in X \times X\}$ is centralized by $\alpha$,  and so $A=\C_G(\alpha)=H \langle \alpha\rangle$
is a subgroup of $G$ with $|A:H|=2$. Note, also,  that $H \cong X$.

Clearly $H$ is a retract of $X \times X$ (for example, under the map $(x,y) \mapsto (x,x)$), so $H \vr G$. However Proposition \ref{prop:wreath_ex} below shows that
$A$ will not be a virtual retract of $G$ unless $X$ is virtually abelian. Thus, for example, if $X$ is the free group of rank $2$ then the subgroup $H$
is a virtual retract of $G=X \wr \langle \alpha \rangle_2$, but its index $2$ overgroup $A$ is not.
\end{ex}

In the next proposition and its proof we will use the notation of Example \ref{ex:wreath}.

\begin{prop}\label{prop:wreath_ex} If $A \vr G$ then $X \cong H$ has an abelian subgroup of finite index.
\end{prop}

\begin{proof} Suppose that there is a finite index subgroup $K \leqslant G$ and a normal subgroup $N \lhd K$ such that $A \subseteq K$, the intersection $A \cap N$ is
trivial  and $K=AN$. Since $|G:(X\times X)|=2$, after replacing $N$ with $N \cap (X \times X)$, we can assume that $N \subseteq X \times X$.
Let $X_1=X \times \{1\} \leqslant G$ and $M=K \cap X_1$. Then $|X_1:M|<\infty$ and we will show that $M$ is abelian.

If $(x,1) \in  N$ for some $x \in X$, then $(x,x)= (x,1) \alpha (x,1) \alpha^{-1} \in A \cap N$, since $A$ contains the diagonal subgroup of $X \times X$ and
$\alpha \in A$ normalizes $N$. Hence $x=1$, and thus
\begin{equation}\label{eq:triv_intersec}
X_1 \cap N=\{(1,1)\} \mbox{ in } X \times X.
\end{equation}

Now, for all $(x,1) \in M$ and $(a,b) \in  N$, the commutator $[(x,1), (a,b)]=([x,a],1)$ belongs to $N$, as
$M \subseteq K$ normalizes $N$. In view of \eqref{eq:triv_intersec}, the latter implies that $[x,a]=1$, i.e.,
\begin{equation}\label{eq:x_comm_with_a}
\mbox{each $(x,1) \in M$ commutes with $(a,1) \in X_1$, provided $\exists~b \in X$ with $(a,b) \in N$. }
\end{equation}

Now, for every $(a,1) \in M \leqslant K=AN$ there must exist $\epsilon \in \{0,1\}$, $(c,c) \in H$ and $(d,e) \in N$ such that $(a,1)=\alpha^\epsilon (c,c)(d,e)$.
Then $\epsilon=0$,  $ce=1$, and so $a=cd=e^{-1}d$. Finally, observe that for $b=d^{-1}e \in X$ we have
\[(a,b)=(e^{-1},d^{-1})(d,e)=\alpha (d,e)^{-1} \alpha^{-1} (d,e) \in N,
\]   as $N$ is normalized by $\alpha \in K$. Recalling \eqref{eq:x_comm_with_a}, we deduce that $(a,1)$ is central in $M$. Hence $M$ is abelian, and so $X_1 \cong X \cong H$
are virtually abelian.
\end{proof}

Proposition \ref{prop:wreath_ex} implies that for $A \vr G$ in Example \ref{ex:wreath}, $X$ must necessarily be virtually abelian.
However this is not always sufficient, as the following exercise shows.

\begin{exercise}\label{exer:ab_wreath} Using the notation of Example \ref{ex:wreath}, suppose that the group $X$ is abelian.
Then $A\vr G$ if and only if $X$ has a finite index subgroup without $2$-torsion and $|X/X^2|<\infty$.
\end{exercise}

\begin{rem} One can easily make torsion-free examples similar to Example~\ref{ex:wreath}. Indeed, let $X$ be a torsion-free group which is not virtually abelian, and let
$G=(X \times X) \rtimes \langle \beta \rangle$ be a semidirect product of $X \times X$ with the infinite cyclic group generated by $\beta$,
where $\beta (x,y) \beta^{-1}=(y,x)$, for all $(x,y) \in X \times X$
(i.e., $\beta$ induces the same involution of $X\times X$ as $\alpha$ did in Example~\ref{ex:wreath}). Let $D=\{ (x,x)\mid x \in X\}$ be the diagonal subgroup of $X \times X$,
let $H=D \langle \beta^2\rangle$ and $A=\C_G(\beta)=D \langle \beta\rangle$. Obviously
$G$ is torsion-free,  $H \cong X \times \Z \cong A$, $|H:A|=2$ and
$H \vr G$ (as $\langle X\times X, \beta^2 \rangle \cong X\times X \times \langle \beta^2\rangle$ has index $2$ in $G$).
However, an argument very similar to the proof of Proposition~\ref{prop:wreath_ex} shows that $A$ is not a virtual retract of $G$.
\end{rem}

\section{Retractions onto virtually abelian subgroups}\label{sec:retr_on_vab}
The aim of this section is to prove Theorem \ref{thm:virt_ab-lift}, which provides a strong converse to Proposition \ref{prop:wreath_ex}.
The argument will make use of a few auxiliary statements.

\begin{lemma} \label{lem:retr_onto_virt_ab->homom} Let $G$ be a residually finite group and let $H \vr G$ be a finitely generated virtually abelian subgroup.
If $A \leqslant G$ contains $H$ and $|A:H|<\infty$
then there is a finitely generated virtually abelian group $P$ and an epimorphism $\varphi:G \to P$ such that $\varphi$ is injective on $A$.
\end{lemma}

\begin{proof} Let $K \leqslant G$ be a finite index subgroup retracting onto $H$. Then $K$ is residually finite, so we can apply
Lemmas \ref{lem:virt_retr->closed} and \ref{lem:closed->sep} to find a finite index subgroup $L \leqslant K$ such that $A\cap L=H$.
Clearly $L$ still retracts onto $H$, so there exists $N \lhd L$ such that $L=HN$ and $H \cap N=\{1\}$. Since $|G:L| =|G:K|\,|K:L|<\infty$,
there is a finite index normal subgroup $M \lhd G$, contained in $L$. Denote $N_1=N \cap M \lhd M$, and observe that $M/N_1$ is finitely generated and virtually abelian as
it is isomorphic to $MN/N$ which has finite index in $L/N \cong H$.

Let $G=\bigsqcup_{i=1}^n g_i M$, where $g_1,\dots, g_n \in G$, $g_1=1$, and $N_2=\bigcap_{i=1}^n g_i N_1 g_i^{-1}$. Then $N_2 \lhd G$ and $|G/N_2:M/N_2|=|G:M|<\infty$.
Since $M \lhd G$ and $N_1\lhd M$, the conjugate  $g_iN_1g_i^{-1}$ is a normal subgroup of $M$, for each $i=1,\dots,n$. Therefore
$M/N_2$ can be naturally (``diagonally'') embedded in the direct product $\times_{i=1}^n M/g_i N_1 g_i^{-1}$. Note that
for every $i \in \{1,\dots, n\}$ the group $M/g_i N_1 g_i^{-1}=g_iMg_i^{-1}/g_i N_1 g_i^{-1} \cong M/N_1$ is finitely generated and virtually abelian,
therefore $\displaystyle \times_{i=1}^n M/g_i N_1 g_i^{-1}$ is also finitely generated and virtually abelian. Since any subgroup of finitely generated virtually abelian group is
again finitely generated and virtually abelian, both $M/N_2$ and its finite index overgroup $P=G/N_2$ are finitely generated and virtually abelian.

It remains to show that the natural epimorphism $\varphi:G \to P$ is injective on $A$, in other words, that $N_2 \cap A=\{1\}$. But this is true by construction, because $N_2\subseteq L$ and $N_2 \subseteq N_1 \subseteq  N$, so $N_2 \cap A \subseteq N \cap (L \cap A) = N \cap H=\{1\}$.
\end{proof}

The next statement is essentially a consequence of Maschke's Theorem from Representation Theory (see \cite[Theorem 10.8]{Curtis-Reiner}).

\begin{lemma} \label{lem:virt_ab_normal_retr} Let $L$ be a finitely generated virtually abelian group and let $S \lhd L$ be a normal subgroup. Then
there is a torsion-free normal subgroup $R \lhd L$ such that $|L:SR|<\infty$ and $S \cap R$ is trivial.
\end{lemma}

\begin{proof} Let $F \lhd L$ be a free abelian normal subgroup of finite index in $L$, thus $F \cong \Z^n$ for some $n \in \N \cup \{0\}$.
Then $T=S \cap F$ is normal in $L$, so the natural action of $L$ on $F$ by conjugation gives a homomorphism $\psi:L \to Aut(F) \cong GL_n(\Z)$, with finite image $X=\psi(L)$ (since $|L:F|<\infty$ and $F$ is abelian). Thus $F$ can be regarded as an $X$-module,
with $T$ -- an $X$-submodule of $F$.

Now, $U=F \otimes \Q$ is a vector space over $\Q$, equipped with the induced action of $X$ and containing an $X$-invariant subspace $V=T \otimes \Q$. By Maschke's Theorem,
$V$ has an $X$-invariant complement $W$, so that $U=V + W$ and $V \cap W$ is trivial. Note that $F$ is an $X$-invariant lattice in $U$, so $R=W \cap F$ is also $X$-invariant.
Consequently, $R \leqslant F$ (so $R$ is torsion-free), $R$ is normal in $L$ and the intersection $T \cap R$ is trivial.
It follows that $S \cap R$ is trivial.

It remains to show that $|P:SR|<\infty$. By the construction, for every $f \in F$, $f=v+w$, where $v\in V$ and $w \in W$.
Obviously there exists a positive integer $m \in \N$ such that $m v \in F$ and $m w \in F$
($m$ is the common denominator for all the rational numbers `involved' in $v$ and $w$), hence $mv \in T$ and $mw \in R$. Thus $mf \in T+R$. Since such $m$ exists for each
$f \in F$, we deduce that $F/(T+R)$ is a torsion abelian group. But this group is also finitely generated (as $F$ is finitely generated), hence $|F/(T+R)|<\infty$.
Since  $|L:F|<\infty$ and $T+R =TR \leqslant SR$, we can conclude that $|L:SR|<\infty$, so the lemma is proved.
\end{proof}

\begin{cor} \label{cor:retr_in_virt_ab_gps} If $P$ is a finitely generated virtually abelian group then every subgroup $Q \leqslant P$ is a virtual retract.
\end{cor}

\begin{proof} By the assumptions, $P$ contains a normal abelian subgroup $F \lhd P$ of finite index.
Then $L=Q F$ is a finite index subgroup of $P$ and $S=Q \cap F$ is normal in $L$. The group $L$ is again finitely generated and virtually abelian, so we can
apply Lemma \ref{lem:virt_ab_normal_retr} to find a torsion-free normal subgroup
$R \lhd L$ which has trivial intersection with $S$ and satisfies $|L:SR|<\infty$. Since $Q \subseteq L$, we see that $R$ is normalized by $Q$.
Observe, also, that
\[|P:QR|\le |P:L|\,|L:QR| \le |P:L|\,|L:SR|<\infty.\]
Finally, $|Q \cap R|<\infty$ as $S$ has finite index in $Q$,
which implies that  $Q\cap R$ is trivial because
$R$ is torsion-free. Hence we have shown that $Q \vr P$, as required.
\end{proof}

We can now prove Theorem \ref{thm:virt_ab-lift}.
\begin{proof}[Proof of Theorem \ref{thm:virt_ab-lift}] By Lemma \ref{lem:retr_onto_virt_ab->homom}, there is an epimorphism $\varphi:G \to P$, where $P$ is
a finitely generated virtually abelian group and
the restriction of $\varphi$  to $A$ is injective.
So we can use Corollary~\ref{cor:retr_in_virt_ab_gps} to deduce that
$\varphi(A) \vr P$. Finally, we can conclude that $A \vr G$ by Lemma \ref{lem:basic_props_vr}.(ii).
\end{proof}

\begin{rem} Examples \ref{ex:rf_neces} and \ref{ex:wreath} show that in Theorem \ref{thm:virt_ab-lift} the assumptions that $G$ is residually finite and $H$ is virtually abelian are indeed necessary. Finally, in view of Exercise~\ref{exer:ab_wreath}, one can use the free abelian group of infinite rank as $X$ in Example \ref{ex:wreath} to see that the finite generation of $H$ is also important.
\end{rem}

As the reader will notice, finitely generated virtually abelian groups play a prominent role in this paper. One explanation for this is that
Corollary \ref{cor:retr_in_virt_ab_gps} has a converse, at least in the case of groups with finite virtual cohomological dimension.

\begin{prop}\label{prop:strong_VR->virt_ab} Let $G$ be a virtually torsion-free group with ${\rm vcd}(G)<\infty$. If every subgroup of $G$ is a virtual retract then
$G$ is finitely generated and virtually abelian.
\end{prop}

\begin{proof} Let $A \leqslant G$ be a torsion-free subgroup of finite index, then the cohomological dimension $\cd(A)$ is equal to ${\rm vcd}(G)<\infty$,
and every subgroup of $A$ is still a virtual retract by Lemma~\ref{lem:basic_props_vr}.(i). Thus it is enough to show that $A$ is finitely generated and virtually abelian.
The argument will proceed by induction on $\cd(A)$. If $\cd(A)=0$ then $A$ is the trivial group 
and so the statement holds. Hence we can
further suppose that $A$ is non-trivial.

Take an infinite cyclic subgroup $C \leqslant A$. By the assumptions, $C \vr A$, so there exist a finite index subgroup $K \leqslant A$ and
$S \lhd K$ such that $C \subseteq K$, $C \cap S=\{1\}$ and $K=C S$.
Now, $S \vr G$, so  $S \vr K$, and one can find a finite index subgroup $L \leqslant K$ and $T \lhd L$ satisfying $S \subseteq L$, $S \cap T=\{1\}$ and $L=ST$.
Note that $S \lhd L$, so $L=ST \cong S \times T$ and $T \cong L/S$. Since $L/S$ has finite index in $K/S \cong C$, we deduce that $T\cong \Z$.

Therefore $L \cong \Z \times S$, and since $\cd(L)=\cd(A)<\infty$ (as $|A:L|<\infty$, see \cite[Chapter~VIII.2, Proposition~2.4.(a)]{Brown}),
we have $\cd(S)=\cd(L)-\cd(\Z)<\cd(A)$ (cf. \cite[Chapter~II, Theorem~5.5]{Bieri}).
Lemma~\ref{lem:basic_props_vr}.(i) tells us that every subgroup of $S$ is a virtual retract, as $S \leqslant A$, so $S$ must be finitely generated and virtually abelian by induction. It follows that $L$, and, hence, $A$ are finitely generated and virtually abelian. Thus the proposition is proved.
\end{proof}

\section{On properties \VR{} and \VRC{}.}\label{sec:VR_and_VRC}
In this section we will establish some basic facts about properties \VR{} and \VRC{} from Definition \ref{def:VR_and_VRC}.
Let us start with proving Proposition~\ref{prop:VRC->VRAb} stated in the Introduction, as the first interesting application of Theorem \ref{thm:virt_ab-lift}.

\begin{proof}[Proof of Proposition \ref{prop:VRC->VRAb}] Note that $G$ is residually finite by Lemma \ref{lem:vrc->rf}.
Since every finitely generated virtually abelian subgroup has a finite index subgroup isomorphic to $\Z^n$, for some $n \in \N\cup \{0\}$, by
Theorem~\ref{thm:virt_ab-lift} it is enough to show that $A \vr G$ for every free abelian subgroup $A$ of rank $n$. We will prove the latter statement by
induction on $n$. If $n \le 1$, it follows from property \VRC{}, so suppose that $A \cong \Z^n$ with $n \ge 2$.

Choose any subgroup $B \leqslant A$, with $B \cong \Z^{n-1}$. Then $B \vr G$ by the induction hypothesis, so there exist a finite index subgroup $K$ of $G$ and $N \lhd K$
such that $B \subseteq K$, $B \cap N$ is trivial and $K=BN$. Note that $|A:(A \cap K)|<\infty$ , so $A \cap K \cong \Z^n$, hence $A \cap N$ must be an
infinite cyclic group, generated by some element $a \in A \cap N$. Evidently $|A:B\langle a\rangle|<\infty$.

By the assumptions, $\langle a \rangle \vr G$, hence $\langle a \rangle \vr K$ by Lemma \ref{lem:basic_props_vr}.(i).
Therefore there exists a finite index subgroup $L \leqslant K$
which contains $a$ and admits a homomorphism $\xi: L \to C$, where $C$ is an infinite cyclic group generated by $\xi(a)$. Denote $B_1=B \cap L$,
then $|B:B_1|<\infty$, so $B_1 \cong \Z^{n-1}$ and $|A:B_1\langle a \rangle|<\infty$.
Let $\rho:L \to B$ be the restriction to $L$ of the retraction of $K$ onto $B$; then $a \in \ker\rho=L \cap N$ and
$\rho$ is injective on $B_1$.

Let $\psi:L \to C \times B$ be the homomorphism defined by $\psi(g)=(\xi(g),\rho(g))$ for all $g \in L$.
Observe that $\psi(\langle a \rangle) \subseteq C \times \{1\}$, as $\rho(a)=1$. On the other hand,
$\psi(B_1)$ intersects  $C \times \{1\}$ trivially in $C \times B$, because the composition of $\psi$ with the projection of $C \times B$ onto $B$ is the homomorphism
$\rho$, which is injective on $B_1$. Therefore $\psi(B_1) \cap \langle \psi(a) \rangle$ is trivial in $C \times B$.

By construction, $\psi$ is injective on $B_1$ and on $\langle a \rangle$, and the images of these two subgroups have trivial intersection. Therefore $\psi$ is injective
on the product $B_1\langle a\rangle$. Also, one can note that $C \times B \cong \Z^n$ and $B_1\langle a\rangle \cong B_1\times \langle a \rangle \cong \Z^n$, hence
$\psi(B_1\langle a \rangle)$ must have finite index in $C \times B$. It follows that for $S=\ker \psi \lhd L$ we have $S \cap B_1\langle a \rangle=\{1\}$ and
$|L:B_1\langle a \rangle S|<\infty$. Consequently, $B_1\langle a \rangle \vr L$, hence $B_1\langle a \rangle \vr G$ as $|G:L|<\infty$. Since $G$ is residually finite
and $|A:B_1\langle a \rangle|<\infty$ we can apply Theorem \ref{thm:virt_ab-lift} to conclude that $A \vr G$, as required.
\end{proof}

\begin{lemma}\label{lem:vr-obs} Let (P) be one of the properties \VR{} or \VRC{}.
\begin{itemize}
  \item[(i)] Suppose that $G$ is a group satisfying (P). Then $G$ is residually finite and every subgroup $A \leqslant G$ also satisfies (P).
  \item[(ii)] If $G$ has \VRC{} then every finitely generated virtually abelian subgroup is closed in the profinite topology on $G$.
  \item[(iii)] If $G$ has \VR{} then $G$ is LERF.
\end{itemize}
\end{lemma}

\begin{proof} Claim (i) follows from Lemmas \ref{lem:vrc->rf} and \ref{lem:basic_props_vr}.(i).
Claims (ii) and (iii) follow from (i) and a combination of Proposition \ref{prop:VRC->VRAb} with Lemma~\ref{lem:virt_retr->closed}.
\end{proof}

\begin{lemma}\label{lem:vrab-basics}
{\rm (a)} If $K$ is a finite index subgroup of a group $G$ and $K$ has \VRC{} then $G$ also has \VRC{}.

{\rm (b)} The (finitary) direct product $\times_{i \in I} X_i$ of any family of groups $X_i$ with \VRC{} satisfies \VRC{}.
\end{lemma}

\begin{proof} Claim (a) is a consequence of Theorem \ref{thm:virt_ab-lift}. Indeed, suppose that
$A \leqslant G$ is a cyclic subgroup and set $H=A \cap K$. Then $|A:H|<\infty$ and $H$ is also cyclic, hence $H \vr K$, so $H \vr G$.
Therefore $A\vr G$ by Theorem \ref{thm:virt_ab-lift} ($G$ is residually finite because $K$ is, see Lemma~\ref{lem:vr-obs}.(i)).

Claim (b) was proved in \cite[Theorem 2.13]{L-R} when $|I|=2$, and our argument is a simple generalization of this proof.
Each group $X_i$, $i \in I$, is residually finite by Lemma \ref{lem:vr-obs}.(i), hence the same is true for their direct product
$P=\times_{i \in I} X_i$.
Suppose that $\langle a \rangle \leqslant P$ is a cyclic subgroup. If $a$ has finite order then $\langle a \rangle \vr P$ by
Lemma~\ref{lem:retract_fin-ext}. On the other hand, if $a$ has infinite order then there must exist some $j \in I$ such that $\rho_j(a)$ has infinite order in $X_j$,
where $\rho_j:P \to X_j$ denotes the canonical projection (this is because $P$ is the \emph{finitary} direct product, so that $a$ has only finitely many
non-trivial projections to $X_i$, $i \in I$). It follows that $\rho_j$ is injective on $\langle a \rangle$. Since $\rho_j(\langle a \rangle) \vr X_j$ by the assumptions,
$\langle a \rangle \vr P$ by Lemma \ref{lem:basic_props_vr}.(ii). Thus we have proved claim (b).
\end{proof}

\begin{cor}\label{cor:v_free-VRC} Free groups have \VR{} and virtually free groups have \VRC.
\end{cor}

\begin{proof} Every finitely generated subgroup of a free group is a free factor of a finite index subgroup by M. Hall's theorem \cite{M-Hall} (cf. \cite[Corollary 1]{Burns1}), so free groups (of arbitrary rank) have \VR{}, and, hence, \VRC{}. Therefore virtually free groups have \VRC{} by Lemma~\ref{lem:vrab-basics}.(a).
\end{proof}

\begin{rem}\label{rem:claim_in_L-R}
Part (a) of Lemma \ref{lem:vrab-basics} was stated as Theorem 2.11 in \cite{L-R}. However, the proof of this theorem in \cite{L-R} refers to
\cite[Theorem 2.10]{L-R}, whose proof is invalid. More precisely, Theorem 2.10 of \cite{L-R} claims that if $G$ is a finitely generated linear group and $K\leqslant G$
is a finite index subgroup satisfying \VR{} then $G$ also satisfies \VR{}. We do not know if this theorem is true as stated (see Question~\ref{q:L-R} below);
Proposition~\ref{prop:f_i_overgp-not_VR} shows that property \VR{} is not, in general, preserved by passing to finite index overgroups, but our example is not linear over $\mathbb{C}$.

The problem with the proof of \cite[Theorem 2.10]{L-R} is as follows (we use the notation from \cite{L-R}).
After the non-faithful linear representation of $A$ is induced to a representation of $AH$, the image of $H$, with respect to the induced representation, will
usually have infinite index in the image of $AH$ (indeed, the kernel of the induced representation can be much smaller than the kernel of the original representation).
This would exactly be the issue if one tried to show that in the linear group $G=F_2 \wr \Z_2$ the centralizer of $\Z_2$ is a virtual retract, by using that the diagonal subgroup of $F_2 \times F_2$ is a virtual retract of $G$: see Example~\ref{ex:wreath} and Proposition~\ref{prop:wreath_ex}.
\end{rem}

\begin{rem}\label{rem:VR_dir_prod} (a) Property \VRC{} may not be preserved by infinite cartesian products. This can be seen, for instance, from the existence
of residually finite groups without \VRC{} (cf. Example~\ref{ex:Heisenberg} below) and the fact that every residually finite group embeds in a
cartesian product of finite groups (each of which certainly has \VRC{}).

(b) Property \VR{} does not behave well even under finite direct products.
Indeed, the free group $F_2$, of rank $2$, satisfies \VR{} by
Corollary~\ref{cor:v_free-VRC}, however it is well-known that the direct product $F_2 \times F_2$ is not LERF
(see, for instance, \cite[Example on p. 12]{All-Greg}), so it cannot satisfy \VR{} by Lemma \ref{lem:vr-obs}.(iii).
\end{rem}

The next proposition implies that \VR{} is at least stable under taking direct products with finitely generated virtually abelian groups.
\begin{prop}\label{prop:vr_times_virt_ab_is_vr} Suppose that $X$ is a finitely generated virtually abelian group and $Y$ is any group satisfying \VR{}.
Then $X \times Y$ also satisfies \VR{}.
\end{prop}

The proposition easily follows from the following lemma.
\begin{lemma}\label{lem:crit_for_vr_in_dir_prod} Let $X$ be a finitely generated virtually abelian group and let $Y$ be an arbitrary group.
Suppose that $H \leqs X \times Y$ is a subgroup such that
$\rho_Y(H) \vr Y$, where $\rho_Y:X \times Y \to Y$ is the natural projection. Then $H \vr X\times Y$.
\end{lemma}

\begin{proof} Let $\rho_X:X\times Y \to X$ denote the natural projection, set $L=\rho_X(H) \leqslant X$ and $M=\rho_Y(H) \leqslant Y$.
Then $M \vr Y$ by the assumptions, and $L \vr X$ by Corollary \ref{cor:retr_in_virt_ab_gps}.
Therefore $L \times M \vr X \times Y$ by Lemma \ref{lem:basic_props_vr}.(v). Note that $H \subseteq L \times M$, so, according to Lemma \ref{lem:basic_props_vr}.(iv),
to show $H \vr X \times Y$ it is sufficient to prove that $H \vr L \times M$.

Observe that the subgroup $S=H\cap L$ is normalized by $H$ (as $L \lhd L \times M$), and, since $H$ projects \emph{onto} $L$, we can deduce that $S \lhd L$
(cf. \cite[Lemma 2.1.(i)]{Min-subdir}). The group $L$ is finitely generated and virtually abelian, as it is a subgroup of $X$, so we can apply
Lemma \ref{lem:virt_ab_normal_retr} to find $R \lhd L$ such that $|L:SR|<\infty$ and
$S \cap R$ is trivial. Then $R \lhd L \times M$, in particular $R$ is normalized by $H$. Moreover, $R \cap H=R \cap (L \cap H)=R \cap S$ is trivial, so it remains to show that
$|(L \times M):HR|<\infty$. Indeed, note that $|L:(L \cap HR)| \le |L: SR|<\infty$ and the subgroup $HR \leqslant L \times M$ still projects \emph{onto} $L$ and $M$. Hence
$|(L \times M):HR|=|L:(L \cap HR)|<\infty$ by \cite[Lemma 2.1.(iii)]{Min-subdir}. Thus $H \vr L \times M$, so $H \vr X \times Y$ and the proof is complete.
\end{proof}

Although \VR{} may fail to pass to finite index overgroups (see Proposition~\ref{prop:f_i_overgp-not_VR}), it is preserved by some quotients and
by extensions with finite kernel.
\begin{lemma} Let $\{1\} \to L \hookrightarrow G \stackrel{\psi}{\twoheadrightarrow} M \to\{1\}$ be a short exact sequence of groups.
\begin{itemize}
  \item[(i)] If $G$ has \VR{} and $L$ is finitely generated then $M$ has \VR{}.
  \item[(ii)] If $M$ has \VR{}, $L$ is finite and $G$ is residually finite then $G$ has \VR{}.
\end{itemize}
\end{lemma}

\begin{proof} (i) Consider any finitely generated subgroup $A \leqslant M$. Since $L$ is finitely generated, $H = \psi^{-1}(A) \leqslant G$ is also finitely generated, hence
$H \vr G$ as $G$ satisfies \VR{}. Thus there exists a subgroup $N \leqslant G$ such that $N$ is normalized by $H$, $H \cap N=\{1\}$ and $|G:HN|<\infty$.

Set $B=\psi(N) \leqslant M$ and observe that $B$ is normalized by $A=\psi(H)$ and $|M:AB|\le |G:HN|<\infty$. Note that
\[\psi^{-1}(A \cap B)=\psi^{-1}(A) \cap \psi^{-1}(B)=H \cap NL=L,\] where the last equality follows from the fact that
$L \subseteq H$, so $H \cap NL=(H \cap N)L=L$. Therefore $A \cap B$ is trivial in $M=G/L$, thus $A \vr M$. Hence $M$ has \VR{}.

(ii) Let $H \leqslant G$ be a finitely generated subgroup. Then $A=\psi(H) \vr M$, since $M$ satisfies \VR{}, so there is $B \leqslant M$ such that $B$ is normalized by $A$,
$A \cap B=\{1\}$ and $|M:AB|<\infty$. It is clear that $N=\psi^{-1}(B) \leqslant G$ will be normalized by $H$, will have finite intersection with $H$ (as $|L|<\infty$)
and will satisfy $|G:HN|=|M:AB|<\infty$.

Since $G$ is residually finite, there is a finite index normal subgroup $K\lhd G$ which intersects $H \cap N$ trivially. After setting $N'=N \cap K$, we see that $N'$
is normalized by $H$, $H \cap N'=\{1\}$ and $|HN:HN'|\le |N:N'|<\infty$, so $|G:HN'|<\infty$. Hence $H \vr G$, and thus $G$  has \VR{}.
\end{proof}

Recall that the \emph{first virtual betti number} of a group $G$ is defined by
\[\vb(G)=\sup\{\bone(K) \mid K \leqslant G,~|G:K|<\infty\},\]
where $\bone(K)$ is the \emph{$\Q$-rank} $\rk(K/[K,K])$ of the abelianization $K/[K,K]$; in other words,
$\bone(K)=\rk(K/[K,K])=\dim_\Q\bigl(K/[K,K] \otimes \Q\bigr)$.

\begin{prop}[{cf. \cite[Theorem 2.14]{L-R}, \cite[Chapter 6, (G.19)]{A-F-W}}]\label{prop:VRC->inf_vbn}
Let $G$ be a virtually torsion-free group with \VRC{}. If $G$ is not virtually abelian then $\vb(G)=\infty$.
\end{prop}

\begin{proof} It is sufficient to prove that for any subgroup of finite index $K \leqslant G$ there is a finite index subgroup $L \leqslant K$ such that
$\bone(L) \ge \bone(K)+1$.

Let $\varphi:K \to A=K/[K,K]$ be the natural homomorphism from $K$ to it abelianization. Since $K$ is virtually torsion-free and is not virtually abelian,
$\ker\varphi=[K,K]$ must be infinite, and there must exist an infinite order element $g \in \ker\varphi$. By the assumptions, $\langle g \rangle \vr G$, hence there is a
finite index subgroup $L \leqslant K$ which contains $g$ and has a homomorphism $\rho:L \to \Z$ such that $\rho(g)=1$ (here $1$ denotes the generator of $\Z$).

Consider the homomorphism $\psi:L \to A \oplus \Z$ defined by $\psi(h)=(\varphi(h),\rho(h))$ for all $h \in L$. Note that $\psi(\langle g \rangle) = \{0\}\oplus\Z$
since $g \in \ker\varphi$. On the other hand, we claim that $\psi(L)$ contains $\varphi(L) \oplus\{0\}$,
which is a finite index subgroup of $A\oplus\{0\}$, since $|K:L|<\infty$. Indeed, for
every $a \in \varphi(L)$ we can take any $h\in \varphi^{-1}(a) \cap L$ and let $n=-\rho(h) \in \Z$; then $\psi(hg^n)=(a,0)$.
It follows that $\psi(L)$ has finite index in $A\oplus \Z$, so $\rk(\psi(L))=\rk(A\oplus \Z)=\rk(A)+1=\bone(K)+1$.

Since $\psi(L)$ is an abelian quotient of $L$, it must also be a quotient of $L/[L,L]$, so $\bone(L)=\rk(L/[L,L])\ge \rk(\psi(L))\ge \bone(K)+1$, as required.
\end{proof}

\begin{rem}
The assumption that $G$ is virtually torsion-free is necessary in Proposition \ref{prop:VRC->inf_vbn}:
by Theorem \ref{thm:wreath->VRAb} below the lamplighter group $\Z_2\wr \Z$ satisfies \VRC{}
but $\vb(\Z_2\wr \Z)=1$, as it is torsion-by-cyclic.
\end{rem}

\section{Amalgams over retracts}\label{sec:amalg_over_retr}
The free product of two groups with \VR{} again has \VR{} by a result of Gitik, Margolis and Steinberg \cite[Theorem 1.5]{G-M-S}. In this section we will show that property
\VRC{} is closed under a more general construction of amalgamated products over retracts.

Let $P$ and $Q$ be groups, let $R$ be a retract of $P$ and let $S$ be a retract of $Q$. Given an isomorphism $\varphi:R \to S$, we can form the amalgamated free product
\begin{equation}\label{eq:amalg_over_retr}
G=P*_{R=S} Q=\langle P,Q \,\|\, r=\varphi(r)~\forall\,r \in R\rangle.
\end{equation}
In this case we will say that $G$ is an \emph{amalgamated product of $P$ and $Q$ over the common retract $R$}. Let $\rho_1:P \to R$ and $\rho_2:Q \to S$ be retractions of $P$
onto $R$ and of $Q$ onto $S$, with kernels $K_1=\ker \rho_1 \lhd R$ and $K_2=\ker\rho_2\lhd Q$. Then $P=K_1 \rtimes R$, $Q=K_2 \rtimes S\cong K_2 \rtimes R$, and the group
$G$, given by \eqref{eq:amalg_over_retr}, is naturally isomorphic to the semidirect product $(K_1*K_2) \rtimes R$, where $R$ normalizes each $K_i$, $i=1,2$ (the action of $R$ on $K_1$ is induced from $P$, and the action of $R$ on $K_2$ comes from the composition of $\varphi$ with the action of $S$ on $K_2$ in $Q$).
This isomorphism was first observed by Boler and Evans \cite{Boler-Evans}, who used it to prove that $G$ is residually finite whenever $P$ and $Q$ are residually finite.
Both decompositions of $G$ as an amalgamated free product and as the semidirect product will be useful for us below.

Given normal subgroups $N_1 \lhd P$, $N_1 \subseteq K_1$, and $N_2 \lhd Q$, $N_2 \subseteq K_2$, since $N_1 \cap R=\{1\}$ and $N_2 \cap S=\{1\}$, the epimorphisms
$P \to \overline{P}=P/N_1$ and $Q \to \overline{Q}=Q/N_2$ naturally give rise to an epimorphism
\begin{equation}\label{eq:psi}
\psi_{N_1,N_2}: G=P*_{R=S} Q \to \overline{G}=\overline{P}*_{\overline{R}=\overline{S}} \overline{Q},
\end{equation}
 where
$\overline{R}$ and $\overline{S}$ are the images of $R$ and $S$ in $\overline{P}$ and $\overline{Q}$ respectively.
Obviously $\overline{R} \cong R$ is still a retract of $\overline{P}$
(the kernel of the retraction is $K_1/N_1$) and $\overline{S} \cong S \cong R \cong \overline{R}$ is still a retract of $\overline{Q}$.
Thus $\overline{G}$ is an amalgamated product of $\overline{P}$ and $\overline{Q}$ over the common retract $\overline{R}$.

Note that $\overline{R}$ normalizes each of the subgroups $\overline{K}_i=K_i/N_i$, $i=1,2$, in $\overline{G}$. The
 following observation stems from the fact that the centralizer of a finite subgroup has finite index in its normalizer.

\begin{rem}\label{rem:basic_quot-decomp} In the notation above, suppose that $\overline{K}_i$ is finite, for $i=1,2$.
Then the group $\overline{G} \cong (\overline{K}_1*\overline{K}_2) \rtimes \overline{R}$ has a finite index subgroup decomposing as the direct product
$(\overline{K}_1*\overline{K}_2)\times \overline{R}_0$, where $\overline{R}_0$ is a finite index subgroup of $\overline{R}$.
\end{rem}

Observe that when $N_1=\{1\}$ and $N_2=K_2$, $\eta_1=\psi_{\{1\},K_2}$ is basically a retraction of $G$ onto $P$, whose kernel is the normal closure of $K_2$ in $G$.
Similarly, the homomorphism $\psi_{K_1,\{1\}}$ gives rise to a retraction $\eta_2:G \to Q$. Let $\eta:G \to P \times Q$ denote the homomorphism defined by
$\eta(g)=(\eta_1(g),\eta_2(g))$ for all $g \in G$; then $\eta$ is injective on both $P$ and $Q$.

We will only require the next lemma in the case when $A$ is cyclic. However, in the general form stated here this lemma may be of independent interest and may have
other applications.

\begin{lemma} \label{lem:am_over_retr-basic_quot} Using the notation above, suppose that $P$ and $Q$ are residually finite,
$G$ is given by \eqref{eq:amalg_over_retr}, and $A \leqslant G$ is a subgroup which contains no non-abelian free normal subgroups.
Then either the homomorphism $\eta:G \to P \times Q$ is injective on $A$ or there exist normal subgroups $N_1 \lhd P$, $N_2 \lhd Q$ such that
$N_i \subseteq K_i$, $|K_i/N_i|<\infty$, $i=1,2$, and the epimorphism $\psi_{N_1,N_2}:G \to \overline G$, from \eqref{eq:psi}, is injective on $A$.
\end{lemma}

\begin{proof} Suppose that $\eta$ is not injective on $A$, so $A\cap \ker \eta \neq \{1\}$. Note that by the definition of $\eta$,
$\ker\eta \cap P=\{1\}$  and $\ker\eta\cap Q=\{1\}$, so its kernel acts freely on the Bass-Serre tree $\cT$,
corresponding to the decomposition of $G$ as the amalgamated free product
\eqref{eq:amalg_over_retr} (the vertex stabilizers are conjugates of $P$ and $Q$ in $G$).
Hence this kernel is free (see \cite[Theorem 8.3]{D-D}), and, since $A$ contains no non-abelian free normal subgroups, $A \cap \ker\eta$
must be infinite cyclic, generated by some $a \in A\cap \ker \eta$.
The element $a$ cannot fix any vertex of $\cT$, so it must act as a hyperbolic isometry with some axis $\ell$.
Now, $A$ normalizes $\langle a \rangle$, so its action on $\cT$ will preserve $\ell$. This gives rise to a homomorphism $\xi:A \to Aut(\ell)$, where $Aut(\ell)$ is the
automorphism group of the simplicial line $\ell$, which is isomorphic to the infinite dihedral group $D_\infty$.
Let $B=\ker\xi \lhd A$, then $B \cap \langle a \rangle=\{1\}$ and $|A:B\langle a \rangle|<\infty$ (because $a$ acts on $\ell$ as a non-trivial translation).
Pick any edge $e$ of $\ell$. Since $B$ fixes all of $\ell$ pointwise, $B \subseteq \St_G(e)=gRg^{-1}$, for some $g \in G$.

Note that $\ker\eta \subseteq \langle K_1,K_2 \rangle\cong K_1*K_2$ (because $\langle K_1,K_2 \rangle \lhd G$),
hence $a \in \langle K_1,K_2\rangle$, $a \neq 1$. In this case it was
proved in \cite[p. 51]{Boler-Evans} that there exist normal subgroups $N_1 \lhd P$ and $N_2 \lhd Q$ such that $N_i \subseteq K_i$, $|K_i:N_i|<\infty$ and
$\psi_{N_1,N_2}(a) \neq 1$ in $\overline{G}$ (where $\overline{G}$ is given by \eqref{eq:psi}). The idea is simple: one can use the residual finiteness of
$P$ and $Q$ to find finite index normal subgroups $M_1 \lhd P$ and $M_2 \lhd Q$ which avoid all the non-trivial elements of $K_1$ and $K_2$, respectively, that
appear in the normal form of $a$ (when $a$ is viewed as an element of the free product $K_1*K_2$); one then sets $N_i=K_i \cap M_i$, $i=1,2$.

Denote $\overline{A}=\psi_{N_1,N_2}(A)$, $\overline{B}=\psi_{N_1,N_2}(B)$ and $\overline{a}=\psi_{N_1,N_2}(a)$ in $\overline{G}$.
By construction $\overline{a} \neq 1$ in $\overline{G}$ and $\psi_{N_1,N_2}$ is injective on $B$, as $B$ is contained in
a conjugate of $R$ in $G$ and $\psi_{N_1,N_2}$ is injective on $R$ by definition. Since $\overline{G}$ is an amalgamated product of $\overline{P}$ and $\overline{Q}$ over the common retract
$\overline{R}$, we have a natural homomorphism $\overline{\eta}: \overline{G} \to \overline{P}\times \overline{Q}$, which fits in the commutative diagram \eqref{eq:cd} below.

\begin{figure*}[ht!]
\begin{equation}\label{eq:cd}
\begin{CD}
G @>\psi_{N_1,N_2}>> \overline{G}\\
@VV\eta V @VV\overline{\eta}V\\
P\times Q @>>> \overline{P}\times\overline{Q}
\end{CD}
\end{equation}
\end{figure*}

Since $a \in \ker\eta$, it follows that $\overline{a} \in \ker\overline{\eta}$. And, just like in the case of $a$, we can use the latter to conclude that the non-trivial element
$\overline{a} \in \overline{G}$ must act as a hyperbolic isometry on the Bass-Serre tree $\overline{\cT}$, corresponding to the decomposition of $\overline{G}$ as an
amalgamated product over $\overline{R}$.
In particular, $\overline{a}$ will have infinite order in $\overline{G}$ and $\langle \overline{a} \rangle \cap h\overline{R}h^{-1}=\{1\}$ for every $h \in \overline{G}$
(because $h\overline{R}h^{-1}$ stabilizes an edge of $\overline{\cT}$).
Hence $ \langle \overline{a} \rangle \cap \overline{B}=\{1\}$ in $\overline{G}$.

We have shown that the homomorphism $\psi_{N_1,N_2}$ is injective on $B$ and on $\langle a \rangle$, and the images of these two subgroups have trivial intersection. Therefore
$\psi_{N_1,N_2}$ must be injective on the product $B \langle a \rangle$.
Thus  $B \langle a \rangle \cap  \ker\psi_{N_1,N_2}=\{1\}$, hence $|\ker\psi_{N_1,N_2} \cap A|<\infty$ as $|A:B \langle a \rangle|<\infty$. But all finite normal subgroups of $A$ are contained
in $B$ because $B=\ker\xi$ and the image of $\xi$ has no non-trivial finite normal subgroups ($\xi(A)$ is an infinite subgroup of $Aut(\ell) \cong D_\infty$,
so either $\xi(A) \cong \Z$ or $\xi(A) \cong D_\infty$).
Consequently, $\ker\psi_{N_1,N_2} \cap A=\{1\}$, so $\psi_{N_1,N_2}$ is injective on $A$, as claimed.
\end{proof}

We are now able to prove that property \VRC{} is stable under taking amalgamated products over retracts.

\begin{thm}\label{thm:vrc-stab_apr} Suppose that $P$, $Q$ are groups with isomorphic retracts $R \leqslant P$ and $S\leqslant Q$,
and $G=P*_{R=S} Q$ is the corresponding amalgamated product. If $P$ and $Q$ have \VRC{} then $G$ also has \VRC{}.
\end{thm}

\begin{proof} Note that the groups $P$ and $Q$ are residually finite by Lemma \ref{lem:vr-obs}.(i).
Let $A \leqslant G$ be a cyclic subgroup, and let
$\eta:G \to P\times Q$ be the homomorphism defined before Lemma \ref{lem:am_over_retr-basic_quot}.
Then $\eta(A)\leqs P\times Q$ is cyclic, so $\eta(A) \vr P\times Q$ by Lemma \ref{lem:vrab-basics}.(b).
If the restriction of $\eta$ to $A$ is injective, then $A \vr G$ by Lemma \ref{lem:basic_props_vr}.(ii).

Therefore, we can assume that $\eta$ is not injective on $A$. This implies, according to Lemma~\ref{lem:am_over_retr-basic_quot}, that  there exist normal subgroups
$N_1\lhd P$, $N_2 \lhd Q$ such that the homomorphism $\psi_{N_1,N_2}:G \to \overline{G}$, given by \eqref{eq:psi},
is injective on $A$, $N_i \subseteq K_i$ and $\overline{K}_i=\psi_{N_1,N_2}(K_i)$ is finite, $i=1,2$.

Now, according to Remark \ref{rem:basic_quot-decomp}, there is a finite index subgroup $\overline{G}_0 \leqslant \overline{G}$,
which is isomorphic to the direct product $(\overline{K}_1* \overline{K}_2)\times \overline{R}_0$, for some finite index subgroup $\overline{R}_0\leqslant \overline{R}$.
Note that $\overline{R}_0\leqs \overline{R} \cong R \leqs P$ satisfies \VRC{} by Lemma \ref{lem:vr-obs}.(i). On the other hand, the free product
$\overline{K}_1* \overline{K}_2$ is virtually free (as $|\overline{K}_i|<\infty$, $i=1,2$), hence it satisfies \VRC{} by Corollary \ref{cor:v_free-VRC}.
Therefore, in view of  Lemma~\ref{lem:vrab-basics}.(b), the direct product $\overline{G}_0$ has \VRC{}. Claim (a) of the same lemma yields that $\overline{G}$ has \VRC{}, as
$|\overline{G}:\overline{G}_0|<\infty$.

Hence $\psi_{N_1,N_2}(A) \vr \overline{G}$. Since $\psi_{N_1,N_2}$ is injective
on $A$ we can, once again, use Lemma~\ref{lem:basic_props_vr}.(ii) to conclude that $A\vr G$. Thus the proposition is proved.
\end{proof}

\begin{ex}\label{ex:VR_not_closed_under_APRs}
In contrast to Theorem~\ref{thm:vrc-stab_apr}, property \VR{} is not, in general, closed under taking amalgamated free products over retracts.
To show this we can use an example suggested by Allenby and Gregorac in \cite{All-Greg}.
Let $F_2$ denote the free group of rank $2$, let $P \cong Q \cong \Z \times F_2$, and let $G$ be the amalgamated product of $P$ and $Q$ over $F_2$.
Obviously $G \cong F_2 \times F_2$, so $G$ is not LERF (cf. \cite[Example on p. 12]{All-Greg}), hence it does not have \VR{} by Lemma~\ref{lem:vr-obs}.(iii).
However, the groups $P$ and $Q$ both satisfy \VR{} by Corollary~\ref{cor:v_free-VRC} and Proposition~\ref{prop:vr_times_virt_ab_is_vr}.
\end{ex}

Theorem~\ref{thm:vrc-stab_apr} implies that \VRC{} is closed under taking free products (when $R$ and $S$ are trivial), but we can
extend this even further to amalgamated products and HNN-extensions over finite subgroups.

\begin{cor}\label{cor:VRC_stable_under_amalg_over_finite} Suppose that $G$ is the fundamental group of a finite graph of groups whose edge groups are
finite and vertex groups satisfy \VRC{}. Then $G$ also satisfies \VRC{}.
\end{cor}

\begin{proof} Let $\cT$ be the Bass-Serre tree corresponding to the given splitting of $G$.
The vertex groups of the splitting are residually finite by Lemma \ref{lem:vrc->rf}, hence the group $G$ is itself residually finite
by \cite[Chapter II.2, Proposition 12]{Serre}. Since the edge groups of the splitting of $G$ are finite and there are finitely many of them,
there is a finite index normal subgroup $K \lhd G$ which has trivial intersection with every (conjugate of) edge group.
It follows that $K$ acts on $\cT$ cocompactly and with trivial edge stabilizers. The Structure Theorem of Bass-Serre Theory
(\cite[Chapter I.5.4, Theorem 13]{Serre} or \cite[Chapter I.4, Theorem 4.1]{D-D}), tells us that $K$ splits as a free product $F*(\ast_{i=1}^n K_i)$, where
$F$ is a free group and each $K_i$ is isomorphic to a finite index subgroup of some vertex group from the original splitting of $G$.
In particular, $K_i$ has \VRC{} for every $i=1,\dots,n$.
Therefore the free product $K=F*(\ast_{i=1}^n K_i)$ satisfies \VRC{} by Corollary~\ref{cor:v_free-VRC} and Theorem~\ref{thm:vrc-stab_apr}.
Thus $G$ has \VRC{} by Lemma~\ref{lem:vrab-basics}.(a).
\end{proof}

\section{Graph products}\label{sec:graph_prod}
Let $\Gamma$ be a simplicial graph with vertex set $V\Gamma$ and edge set $E\Gamma$, and let $\mathfrak{G}=\{G_v \mid v \in V\Gamma\}$ be a collection of groups.
The \emph{graph product} $\GG$ is defined as the group obtained from the free product $\ast_{v \in V\Gamma} G_v$ by imposing the relations
\[ [g_u,g_v]=1 ~\mbox{ for all } g_u \in G_u, g_v \in G_v \mbox{ whenever } (u,v) \in E\Gamma.\]

Graph products were originally introduced by E. Green \cite{Green} and naturally generalize free and direct products of groups.
Prominent examples of graph products include \emph{right angled Artin groups}
(when $G_v \cong \Z$ for all $v \in V\Gamma$) and \emph{right angled Coxeter groups} (when $G_v \cong \Z_2$ for all $v \in V\Gamma$). For more background
the interested reader is referred to \cite[Subsection 2.2]{Ant-Min}.

Given a graph product $G=\GG$ and any subset $U \subseteq V\Gamma$, the subgroup $G_U=\langle G_u \mid u \in U\rangle $ is called a \emph{full subgroup of $G$}.
This subgroup is naturally isomorphic to the graph product $\Gamma_U\mathfrak{G}_U$, where $\Gamma_U$ is the full (induced) subgraph of $\Gamma$ on the vertices from $U$ and
$\mathfrak{G}_U=\{G_u \mid u \in U\}$ (see \cite[Section 3]{Ant-Min}). In particular, every group $G_v$, $v \in V\Gamma$, naturally embeds in $G=\GG$
(as the full subgroup $G_{\{v\}}$).

For any full subgroup $G_U \leqslant G=\GG$,  there is a \emph{canonical retraction} $\rho_U:G\to G_U$,
which is defined as the identity map on each $G_u$, $u \in U$, and $\rho_U(g)=1$ for all $g \in G_v$ with $v \in V\Gamma \setminus U$.

When the graph $\Gamma$ is finite, the graph product of groups $G=\GG$ can be constructed from the vertex groups iteratively, by taking amalgamated free
products over retracts. Indeed,
for a vertex $v \in V\Gamma$, let $\link(v) \subseteq V\Gamma$ denote the set of all vertices of $V\Gamma\setminus\{v\}$ adjacent to $v$.
From the defining presentation of $G=\GG$,
we see that it naturally decomposes as the amalgamated product of its full subgroups:
\begin{equation}\label{eq:decomp_of_GG}
G \cong G_A *_{G_C} G_B, ~\mbox{ where } A=V\Gamma\setminus\{v\},~C=\link(v) \mbox{ and } B= C\cup \{v\}.
\end{equation}

\begin{thm}\label{thm:gp_of_vrc} Let $\Gamma$ be a finite graph, let $\mathfrak{G}=\{G_v\mid v \in V\Gamma\}$ be a collection of groups.
If each of the groups $G_v$, $v \in V\Gamma$, satisfies \VRC{} then so does the graph product $G=\GG$.
\end{thm}

\begin{proof} We will use induction on $|V\Gamma|$. The base of induction, when $|V\Gamma| \le 1$, obviously holds, so let us assume that $|V\Gamma| \ge 2$ and the statement
has been proved for all graph products over graphs with fewer vertices.

Take any vertex $v \in V\Gamma$. If $v$ is adjacent to every other vertex in $V\Gamma$, then, by the definition of a graph product,
$G\cong G_v \times G_A$, where $A=V\Gamma\setminus\{v\}$. By the assumptions and the induction hypothesis,
both $G_v$ and $G_A$ satisfy property \VRC{}, so we can apply
Lemma \ref{lem:vrab-basics}.(b) to conclude that $G$ also satisfies this property.

Thus we can suppose that $C= \link(v) \subsetneqq V\Gamma\setminus \{v\}$. Then both $A=V\Gamma\setminus\{v\}$ and $B=C \cup \{v\}$ are proper subsets of $V\Gamma$,
so $G_A$ and $G_B$ satisfy \VRC{} by the induction hypothesis. Recalling the decomposition \eqref{eq:decomp_of_GG} of $G$, and the fact that
$G_C$ is a canonical retract of both $G_A$ and $G_B$, we can use Theorem~\ref{thm:vrc-stab_apr} to deduce that $G$ has \VRC{}, as required.
\end{proof}

Since every finitely generated subgroup of a graph product $G=\GG$ is contained in a full subgroup $G_U$, for some finite subset $U \subseteq V\Gamma$
(see \cite[Remark 3.1]{Ant-Min}), and $G_U \vr G$, one can combine Theorem~\ref{thm:gp_of_vrc} with Lemma~\ref{lem:basic_props_vr}.(iv) to deduce that
the conclusion of this theorem also holds when $\Gamma$ is infinite.

\begin{cor}\label{cor:gp_stability} Property \VRC{} is preserved under taking arbitrary graph products.
\end{cor}

Corollary \ref{cor:virt_spec->vrc} from the Introduction is a direct consequence of Theorem~\ref{thm:gp_of_vrc}, Lemma~\ref{lem:vr-obs}.(i) and
Lemma~\ref{lem:vrab-basics}.(a).

\section{Virtual retractions and quasi-potency}\label{sec:quasipot}
In this section we discuss an application of property \VRC{} to quasi-potency, which was originally introduced by Evans \cite{Evans} as a condition for showing that an amalgamated free product of cyclic subgroup separable groups is also cyclic subgroup separable.

\begin{defn}\label{def:quasi-pot}
Let $G$ be a group. An infinite order element $g \in G$ will be called \emph{quasi-potent} in $G$
if there exists $n \in \N$ satisfying the following property. For every $k \in \N$ there is a finite
group $M$ and a homomorphism $\psi:G \to M$ such that the order of $\psi(g)$ in $M$ is precisely $kn$.

A group $G$ will be called \emph{quasi-potent} if every infinite order element $g \in G$ is quasi-potent.
\end{defn}

The terminology from Definition \ref{def:quasi-pot} is due to Ribes and Zalesskii \cite{Rib-Zal},
but the concept itself appeared much earlier in the work of Evans \cite{Evans}, who used the terms ``$G$ has regular quotients at $g$'' and ``$G$ has regular quotients''.
It was proved independently in \cite{Evans}, \cite{Rib-Seg-Zal} and \cite{Bur-Mar} that the amalgamated product of two cyclic subgroup separable quasi-potent groups along a cyclic subgroup is again cyclic subgroup separable and quasipotent. Note that quasi-potency is important here: Rips \cite{Rips} constructed an example of an amalgamated product two LERF groups over a cyclic subgroup which is not even residually finite (finitely generated examples were later obtained by Wilson and Zalesskii in \cite{Wil-Zal}).

Another use for quasi-potency was discovered by Tang \cite{Tang}, who applied it to prove conjugacy separability for amalgamated products
of virtually free groups and virtually nilpotent groups with unique root property along cyclic subgroups.
Further results in this direction were obtained by Ribes and Zalesskii \cite{Rib-Zal} and Ribes, Segal and Zalesskii \cite{Rib-Seg-Zal}.

\begin{lemma}\label{lem:VRC->quasi-pot} Let $G$ be a group and let $g \in G$ be an element of infinite order
such that $\langle g \rangle \vr G$. Then $g$ is quasi-potent in $G$.
\end{lemma}

\begin{proof} Arguing as in Lemma \ref{lem:retr_onto_virt_ab->homom}, we can find a finitely generated virtually abelian group $P$ and a homomorphism
$\varphi:G \to P$ such that $h=\varphi(g)$ has infinite order in $P$ (residual finiteness of $G$ is not required here, as we only need
the homomorphism to be injective on $\langle g \rangle$ and not on its finite index overgroup).

Since finitely generated virtually abelian groups are quasi-potent (cf. \cite[Lemma 2.2]{Tang} or \cite[Theorem 5.5]{Bur-Mar}), there is $n \in \N$ such that for any $k \in \N$
there exist a finite group $M$ and a homomorphism $\eta: P \to M$ such that the order of $\eta(h)$ is exactly $kn$ in $M$. Hence $\psi=\eta\circ\varphi:G \to M$
is a homomorphism sending $g$ to an element of order $kn$ in $M$. Thus $g$ is quasi-potent in $G$.
\end{proof}


\begin{cor}\label{cor:VRC->quasipot} Every group with \VRC{} is quasi-potent.
\end{cor}

Since groups with \VRC{} are cyclic subgroup separable (Lemma \ref{lem:vr-obs}.(ii)), Corollary \ref{cor:VRC->quasipot} can be combined with
a result of Evans \cite[Theorem 3.4]{Evans} to give the following.

\begin{cor}\label{cor:amalg_over_cyclic} Let $P$ be a group with \VRC{} and let $Q$ be cyclic subgroup separable. If $R \leqslant P$ and $S \leqslant Q$ are
infinite cyclic then $P*_{R=S} Q$ is cyclic subgroup separable.
\end{cor}

If one assumes that both $P$ and $Q$ have \VRC{}, then a theorem of Burillo and Martino \cite[Theorem 3.6]{Bur-Mar} gives an even stronger statement.

\begin{cor}\label{cor:amalg_over_virt_cyclic} Suppose that $P$, $Q$ are groups satisfying \VRC{}, and $R \leqslant P$, $S \leqslant Q$ are
isomorphic virtually cyclic subgroups. Then the amalgamated product $P*_{R=S} Q$ is cyclic subgroup separable and quasi-potent.
\end{cor}

\begin{rem} In view of Proposition \ref{prop:VRC->VRAb}, it would be natural to ask whether one can replace `virtually cyclic' by
`finitely generated virtually abelian' in Corollary \ref{cor:amalg_over_virt_cyclic}. Unfortunately this would be fruitless as there exist
amalgamated products of two finitely generated virtually abelian groups which are not residually finite (see \cite[Remark~11.6]{Leary-Min}).
\end{rem}

Combining together Corollaries \ref{cor:amalg_over_virt_cyclic} and \ref{cor:virt_spec->vrc} we obtain the following.

\begin{cor}\label{cor:virt_spec_over_virt_cyclic}
Let $P$, $Q$ be groups each of which has a finite index subgroup that embeds into a right angled Artin or Coxeter group.
Then the amalgamated product of $P$ and $Q$ along a virtually cyclic subgroup is cyclic subgroup separable.
\end{cor}

\section{Virtual retracts in solvable groups}\label{sect:solv}
Proposition \ref{prop:vr_times_virt_ab_is_vr} and Lemma \ref{lem:vrab-basics} tell us that properties  \VR{} and \VRC{} are stable under taking direct product
with finitely generated virtually abelian groups, so it is natural to ask whether this can be extended to (split) extensions with finitely generated abelian kernels.
The next result provides a strong negative answer.

\begin{prop}\label{prop:vpolyc+vrc->vab} A virtually polycyclic group satisfies \VRC{} if and only if it is virtually abelian.
\end{prop}

\begin{proof} The sufficiency is given by Corollary \ref{cor:retr_in_virt_ab_gps}, so we only need to prove the necessity.
So, assume that $G$ is virtually polycyclic and has \VRC{}.
Let $M$ be a torsion-free polycyclic subgroup of $G$ of finite index, whose derived length $d$ is the smallest possible
(among all such subgroups of $G$). Then for every finite index subgroup $L \leqslant M$, $L$ must also have derived length $d$, so $L^{(d-1)}$, the $(d-1)$-th member of the derived series of $L$, is a finitely generated free abelian group. Choose such an $L$ so that $L^{(d-1)}$ has the smallest possible rank.

Let us show that $d \le 1$.  Suppose, on the contrary, that $d \ge 2$, and pick any non-trivial element $g \in L^{(d-1)}$.
By the assumptions, $\langle g \rangle \vr G$, so $\langle g \rangle \vr L$ (see Lemma \ref{lem:basic_props_vr}.(i)). Thus there is a finite index subgroup
$K \leqslant L$ which retracts onto $\langle g \rangle$. It follows that the infinite cyclic subgroup $\langle g \rangle$ has trivial intersection with the derived subgroup
$[K,K]$, of $K$.

Recall that $K$ has derived length $d \ge 2$ by the assumptions, so $K^{(d-1)} \subseteq [K,K]$, hence $K^{(d-1)} \cap \langle g \rangle =\{1\}$. On the other hand,
$K^{(d-1)} \subseteq L^{(d-1)}$ as $K \subseteq L$, therefore the rank of the free
abelian group $K^{(d-1)}$ must be strictly smaller than the rank of $L^{(d-1)}$, contradicting the choice of $L$.

Thus we have shown that $d \le 1$, so $L$ is abelian and $G$ is virtually abelian.
\end{proof}

\begin{ex}\label{ex:Heisenberg} The integral Heisenberg group $H$ (the group of all $3\times 3$ unitriangular matrices with integer coefficients) is
nilpotent of class $2$, is isomorphic to a split extension of $\Z^2$ by $\Z$ and is not virtually abelian.
Therefore, by Proposition \ref{prop:vpolyc+vrc->vab},
$H$ does not satisfy \VRC{}, even though it is LERF (see \cite[Chapter 1.C, Exercise 11]{Segal}).
\end{ex}

\begin{rem}\label{rem:v_polyc+VRC}
 (a) One can note that the proof of Proposition \ref{prop:vpolyc+vrc->vab} actually applies more generally to
solvable groups which are virtually torsion-free and all of whose abelian subgroups have finite $\Q$-ranks. In particular, it shows that a strictly
ascending HNN-extension of $\Z^n$, for any $n \in \N$, cannot have \VRC{}.

(b) Proposition \ref{prop:vpolyc+vrc->vab} can also be deduced  from Proposition~\ref{prop:VRC->inf_vbn}, as the first virtual betti number of a polycyclic group is bounded by
its Hirsch length. In fact, by a much stronger result of Bridson and Kochloukova \cite[Theorem A]{Brid-Koch}, $\vb(G)<\infty$ for any finitely presented
nilpotent-by-abelian-by-finite group $G$. Thus, in view of Proposition~\ref{prop:VRC->inf_vbn}, such a group $G$ cannot have \VRC{}, provided it is
virtually torsion-free and not virtually abelian.
\end{rem}

However, in general solvable groups may satisfy \VRC{} and even \VR{} without being virtually abelian. The natural examples to look at are restricted wreath products of abelian groups. Let $A$ and  $B$ be two groups. By the definition, the \emph{restricted wreath product} $A \wr B$ is the semidirect product $A^B \rtimes B$, where
$A^B$ is the set of all functions from $B$ to $A$ with finite supports, equipped with pointwise multiplication coming from the multiplication on $A$.
And for every $b \in B$ and $f \in A^B$, $f:B \to A$, the function $f^b=bfb^{-1}$ is defined by  $f^b(x)=f(xb)$ for all $x \in B$.

The subgroup of $A \wr B$ consisting of all functions $f\in A^B$ such that $f(b)=1$ for each $b \in B\setminus\{1\}$ is clearly isomorphic to $A$. Thus both
$A$ and $B$ can be thought of as subgroups of $A \wr B$. Note that for any subgroup $D \leqslant A$ the restricted wreath product $D \wr B$ naturally embeds
into $A \wr B$, as the subgroup generated by $D^B$ and $B$.

By a result of Gruenberg \cite[Theorems 3.1,3.2]{Gruenberg}, the restricted wreath product $A \wr B$ of two residually finite groups $A$, $B$
is residually finite if and only if either $B$ is finite or $A$ is abelian.
We can use this theorem to characterize wreath products with property \VRC{} in a similar fashion.

\begin{thm}\label{thm:wreath->VRAb} Let $A$  and $B$ be groups and let $G=A \wr B$ be their restricted wreath product.
Then $G$ has \VRC{} if and only if all of the following conditions hold:
\begin{itemize}
  \item[(i)] both $A$ and $B$ have \VRC{};
  \item[(ii)] either $B$ is finite or $A$ is abelian.
\end{itemize}
\end{thm}

\begin{proof}
We first show the necessity. Given that $G$ has \VRC{}, we know that the same holds for $A$ and $B$, as both of them are subgroups of $G$.
Since this property implies that $G$ is residually finite (Lemma \ref{lem:vr-obs}.(i)), claim (ii) follows from Gruenberg's result \cite[Theorem~3.1]{Gruenberg}.

Thus it remains to prove the sufficiency. So, assume that $A$, $B$ have \VRC{}. If $|B|<\infty$ then $A^B$ is the direct product of finitely many copies of $A$,
so it has \VRC{} by Lemma~\ref{lem:vrab-basics}.(b).
Moreover, in this case $|G:A^B|=|B|<\infty$, consequently $G$ has \VRC{} by Lemma~\ref{lem:vrab-basics}.(a).

Therefore we can further suppose that $A$ is abelian (in view of (ii)), and so we will use the additive notation for the operations on $A$ and $A^B$.
Let $H \leq G$ be a cyclic subgroup, and let $\rho:G \to B$ be the natural retraction, with $\ker\rho=A^B$. We need to consider two cases.

\underline{Case 1:} $\rho(H)$ is a finite subgroup of $B$. Then the subgroup $F=H \cap A^B$ has finite index in $H$.
The subgroup $F$ is generated by a single function from $A^B$, so there is a finite subset $S \subseteq B$ such that  $F \subseteq A^S$ (here by $A^S$ we mean the
subgroup of $A^B$ consisting of functions that are supported on $S$).
Since $B$ is residually finite (see Lemma~\ref{lem:vr-obs}.(i)), we can find a finite index subgroup $C \leqslant B$ which has trivial intersection with the finite subset
$\{s^{-1}t \mid s,t \in S\} \subseteq B$. Observe that  $M=\rho^{-1}(C)=A^B \rtimes C$ is a finite index subgroup of $G$ containing $F$.

Note that the cosets $\{sC \mid s \in S\}$ are all distinct by the choice of $C$, so we can complete $S$ to a finite set $T \subseteq B$ of left coset representatives in
$B/C$. It  is not difficult to see that $M \cong A^T \wr C$, which retracts onto $A^T$ because $A$ is abelian
(the kernel of this retraction is the normal closure of $C$).
More explicitly, let $\eta: M \to A^T$ be the map sending $C$ to the identity and any function $f: B \to A$ to the function $f': T \to A$ defined by
$f'(t)=\sum_{c \in C} f(tc)$ for all $t \in T$. This sum makes sense since $f$ is non-trivial for only finitely many elements of $B$.

Let us check that $\eta$ is a group homomorphism. Indeed, every element of $M$ can be uniquely written in the form $fc$, where $f \in A^B$ and $c \in C$. So,
consider two elements $f_1c_1,f_2 c_2 \in M$, where $f_i \in A^B$, $c_i \in C$, $i=1,2$. Observe that for all $t \in T$ we have
\begin{equation}\label{eq:prod_of_etas}
 \bigl( \eta(f_1c_1)+ \eta(f_2c_2) \bigr)(t)=\eta(f_1c_1)(t)+ \eta(f_2c_2)(t)=f_1'(t)+f_2'(t)=\sum_{c \in C} f_1(tc)+ \sum_{c \in C} f_2(tc).
\end{equation}
On the other hand, for every $t \in T$ we have
\begin{multline}\label{eq:eta_of_prod}
 \eta\bigl( f_1c_1 \, f_2c_2\bigr)(t)= \eta\bigl( (f_1+f_2^{c_1}) (c_1c_2)\bigr)(t)  =(f_1+f_2^{c_1})'(t)=\sum_{c \in C}(f_1+f_2^{c_1})(tc)  \\
= \sum_{c \in C} f_1(tc) +\sum_{c \in C} f_2^{c_1}(tc)=\sum_{c \in C} f_1(tc) +\sum_{c \in C} f_2(tcc_1)=\sum_{c \in C} f_1(tc) +\sum_{c \in C} f_2(tc),
\end{multline}
where we used the facts that $A$ is abelian and $cc_1$ runs over all of $C$ whenever $c$ does.

Comparing the right-hand sides of \eqref{eq:prod_of_etas} and \eqref{eq:eta_of_prod}, we see that $\eta(f_1c_1)+\eta(f_2c_2)= \eta ( f_1c_1 \, f_2c_2 )$, thus
$\eta$ is a homomorphism. Finally, recall that $A^T$, as a subgroup of $A^B$, consists of functions supported on $T$ only, so for any $f \in A^T$,  $f'=f$ as $f(tc)=1$ unless $c=1$. Hence $\eta$ induces the identity map on $A^T$, in
other words $\eta:M \to A^T$ is a retraction. Consequently, $A^T \vr G$.

Now, $A^T$ satisfies \VRC{} by Lemma \ref{lem:vrab-basics}.(b), and $F \leqslant A^S \leqslant A^T$ is a cyclic subgroup.
Therefore
$F \vr A^T$, so $F \vr G$ by Lemma~\ref{lem:basic_props_vr}.(iv). Recall that $A$ and $B$ are residually finite (by Lemma~\ref{lem:vr-obs}.(i)) and $A$ is abelian,
so, according to \cite[Theorem~3.2]{Gruenberg} $G=A \wr B$ is residually finite. Since $|H:F|<\infty$ we can apply Theorem~\ref{thm:virt_ab-lift} to conclude that
$H \vr G$ in Case 1.

\underline{Case 2:} $\rho(H)$ is an infinite subgroup of $B$. It follows that $H$ is infinite cyclic and the restriction of $\rho$ to $H$ is necessarily injective.
By the assumptions, $\rho(H) \vr B$, therefore we can apply Lemma \ref{lem:basic_props_vr}.(ii) to deduce that $H \vr G$ in Case 2.

We have established the sufficiency, so the proof of the theorem is complete.
\end{proof}

Theorem \ref{thm:wreath->VRAb} together with Corollary~\ref{cor:retr_in_virt_ab_gps} tell us that for any finitely generated abelian group $A$,
the wreath product $G=A \wr \Z$ has property \VRC. Moreover, such $G$ is known to be LERF (see \cite[Proposition 3.19]{Cornulier}),
so it is natural to ask whether $G$ also satisfies the stronger property \VR{}. The remainder of the section is devoted to giving a precise answer to this question.
The following result is due to Davis and Olshanskii \cite{Dav-Ols}.

\begin{lemma}\label{lem:Dav-Olsh} Suppose that $G={\Z_p}^k \wr \Z$, where $\Z_p$ is the cyclic group of prime order $p$ and $k \in \N$.
Then $G$ has \VR{}. Moreover, if $H \leqslant G$ is a finitely generated subgroup that is not contained in the normal subgroup $W=({\Z_p}^k)^\Z \lhd G$
then there exists $N \leqslant W$ such that
$N$ is normalized by $H$, $N \cap H$ is trivial and $|G:HN|<\infty$.
\end{lemma}

\begin{proof} Let $H$ be a finitely generated subgroup of $G$. If $H \subseteq W$ then $H$ is abelian and finite, so $H \vr G$
by Lemma~\ref{lem:retract_fin-ext}. Otherwise, $fb^l \in H$  for some $f \in W$
and $l \in \N$, where $b$ denotes the generator of the acting infinite cyclic group (i.e., $G=W \rtimes \langle b \rangle$). In this case the subgroup
$G_l=\langle W, fb^l \rangle \leqslant G$ has index $l$ in $G$ and is naturally isomorphic to the restricted wreath product ${\Z_p}^{kl} \wr \Z$
(cf. \cite[Lemmas 12.4 and 12.5]{Dav-Ols}). Now, \cite[Lemma 10.3]{Dav-Ols} implies that $H$ has finite index in a retract $F$ of $G_l$. Hence $H \vr G$ by
Lemma~\ref{lem:basic_props_vr}.(iv).

The last statement is an immediate consequence of the proof of \cite[Lemma 10.3]{Dav-Ols}
($N$ is the kernel of the retraction $G_l \to F$; $N \subseteq W$ by construction, as it is the normal closure in $G_l$ of finitely many elements from $W$).
\end{proof}

In fact, in \cite[Theorem 1.2]{Dav-Ols} Davis and Olshanskii studied the distortion of subgroups in $G=A \wr \Z$, where $A$ is finitely generated and abelian, and showed that
all finitely generated subgroups of $G$ are undistorted if and only if $A$ is finite. Thus, in view of Remark \ref{rem:vr->undist}, $G$ cannot have \VR{} if $A$ is infinite.
Lemma~\ref{lem:Dav-Olsh} above shows that when $A$ is finite the situation is more interesting, and the next lemma introduces an obstruction not coming from distortion.

\begin{lemma}\label{lem:Z_pm_wr_Z} Let $G=A \wr B$, where $B$ is infinite cyclic  and  $A=\Z_{p^m}$ is the cyclic group of order $p^m$ such that
$p$ is a prime and $m \ge 2$. Then $G$ does not have \VR.
\end{lemma}

\begin{proof} We will think of $\Z_{p^m}$ as the set of residues $\{\overline{0},\overline{1},\dots,\overline{p^m-1}\}$ modulo $p^m$.
By definition, $G=W \rtimes B$, where $B=\langle b \rangle$ is the infinite cyclic group (under multiplication) and
$W$ is the set of all functions $f: B \to \Z_{p^m}$ with finite supports, under modular addition.
Let $h: B \to A$ be the function defined by $h(1)=\overline{p}$ and $h(x)=\overline{0}$ if $x \in B \setminus\{1\}$, and set $H=\langle b,h \rangle \leqslant G$. We will prove that $H$ is not a virtual retract of $G$.

It is easy to see that $H \cap W=pW$, where $pW=\{pf \mid f \in W\}$.
Therefore $W/(H \cap W) \cong  {\Z_p}^\Z$ is infinite, so $H$ has infinite index in $G$.

Arguing by contradiction, suppose that $H \vr G$. Then there exists a subgroup $N \leqslant G$, which is normalized by $H$ and has trivial intersection with it,
and such that $|G:HN|< \infty$.
Since $|G:H|=\infty$, $N$ must be infinite,
so there is a non-trivial element $fc \in N$, where $f \in W$ and $c \in B$. If $c \neq 1$, then the commutator $[fc,h]=h^ch^{-1}$
will be a non-trivial element from $N \cap W$. If $c=1$ then $fc=f \in N \cap W$ will be non-trivial. Thus, there must
exist at least one non-trivial element $g \in N \cap W$.

Since $W \cong {\Z_{p^m}}^\Z$  and $H \cap W=pW$, either $g \in H \cap W$ or $g$ has order $p^m$. In the latter case
$pg \in H \cap W$ has order $p^{m-1}$, hence it is non-trivial as $m \ge 2$. In any case we found a non-trivial element from $N \cap H$, contradicting our assumption.
Therefore $H$ cannot be a virtual retract of $G$; thus $G$ does not have \VR.
\end{proof}

We will call an abelian group \emph{semisimple} if it is the direct sum of cyclic groups of prime order. Thus, for example, ${\Z_2}^2\oplus\Z_3$ is semisimple, while
$\Z_4$ is not.

\begin{thm}\label{thm:vr_for_wreath_with_Z} Suppose that $G=A \wr \Z$, where $A$ is a finitely generated abelian group.
Then $G$ satisfies property \VR{} if and only if $A$ is semisimple.
\end{thm}

\begin{proof} Let us start with proving the necessity. If $A$ is infinite, then $G$ contains finitely generated distorted subgroups by
\cite[Theorem~1.2]{Dav-Ols}, so it does not satisfy \VR{} by
Remark \ref{rem:vr->undist} (alternatively, one can argue that $\Z \wr \Z$ does not have \VR{} similarly to Lemma \ref{lem:Z_pm_wr_Z}, and
since it embeds in $G$ when $A$ is infinite, $G$ cannot have \VR{} either).

Thus we can suppose that $A$ is finite. Then it decomposes as a direct sum of cyclic subgroups of prime-power orders. If one of these subgroups is isomorphic to
$\Z_{p^m}$, for some prime $p$ and $m \ge 2$, then $G$ will contain a copy of $\Z_{p^m} \wr \Z$.
The latter does not have \VR{} by Lemma \ref{lem:Z_pm_wr_Z}, hence $G$ will not have \VR{} either (see Lemma \ref{lem:vr-obs}.(i)).
Therefore, if $G$ satisfies \VR{}, $A$ must be a finite semisimple abelian group.

It remains to prove the sufficiency, so assume that $A=\oplus_{i=1}^n A_i$, where $A_i\cong {\Z_{p_i}}^{k_i}$, $p_1,\dots,p_n$ are pairwise distinct primes and
$k_1,\dots,k_n \in \N$. Let $B$ denote the infinite cyclic acting group for the decomposition of $G$ as a wreath product, i.e., $G=A \wr B$.
Observe that the group $W=A^B$ also splits as a direct sum $W=\oplus_{i=1}^n W_i$, where $W_i$ is the $p_i$-torsion subgroup of $W$; in other words, $W_i={A_i}^B$ is the normal closure of $A_i$ in $G$, $i=1,\dots,n$.

Let $H \leqslant G$ be a finitely generated subgroup. If $H \subseteq W$ then $H$ is abelian and finite, so $H \vr G$ (for example, by
Lemma~\ref{lem:retract_fin-ext} and \cite[Theorem~3.1]{Gruenberg}).
Thus we can further assume that the image of $H$ under the natural retraction $G \to B$ is $\langle b^l \rangle$ for some $l \in \N$, where $b$ is a generator of $B$.
Consequently, $fb^l  \in H$, for some $f \in W$.

Denote $M=H \cap W$, and observe that
$H\cong M \rtimes \langle fb^l\rangle$ and $M=\oplus_{i=1}^n M_i$, where $M_i=M \cap W_i$, $i=1,\dots,n$
(this holds because every element $g \in M$ decomposes as a sum of elements of orders $p_1,\dots,p_n$ from $\langle g \rangle \leqslant M$, and every element of
order $p_i$ in $W$ belongs to $W_i$, $i=1,\dots,n$). For each $i \in \{1,\dots,n\}$ set $G_i=\langle W_i,b \rangle \leqslant G$.
Then $G_i \cong A_i \wr B \cong {\Z_{p_i}}^{k_i}\wr \Z$, and there is a natural retraction $\rho_i:G \to G_i$ whose kernel is the sum of $W_j$ for all $j\neq i$.

Let $f_i=\rho_i(f) \in W_i$, then $\rho_i(fb^l)=f_ib^l$, $i=1,\dots,n$. Observe that for each $i=1,\dots,n$, $H_i=\rho_i(H) = M_i \rtimes \langle f_ib^l \rangle$ is a finitely generated subgroup of $G_i$, which is not contained in $W_i$. So,
by Lemma~\ref{lem:Dav-Olsh}, there exists $N_i \leqslant W_i$ which is normalized by $H_i$, $H_i \cap N_i$ is trivial and $|G_i:H_i N_i|<\infty$.
It follows that $N_i$ is normalized by $f_ib^l \in H_i$, and hence it is also normalized by $b^l$, as $f_i \in W_i$ and $W_i$ is abelian.
The remaining properties of $N_i$ imply that $N_i \cap M_i$ is trivial and $|W_i:(M_i+N_i)|<\infty$, $i=1,\dots, n$.

Now, set $N=\langle N_1,\dots,N_n\rangle=N_1+\dots +N_n \leqslant W$. Then $N$ is normalized by $b^l$, hence it is also normalized by $\langle M, fb^l \rangle=H$.
The construction also implies that the intersection $M \cap N$ is trivial and $M+N$ has finite index in $W$. Therefore $H \cap N$ is trivial and
the subgroup $HN=(M+N) \rtimes \langle fb^l \rangle$ has finite index in $W \rtimes \langle fb^l \rangle=W \rtimes \langle b^l \rangle$, which has index $l$
in $G=W \rtimes \langle b \rangle$. Thus $|G:HN|<\infty$, so $H \vr G$. Therefore  $G$ satisfies \VR{}, as claimed.
\end{proof}

We are now going to prove Proposition \ref{prop:f_i_overgp-not_VR} from the Introduction.
\begin{proof}[Proof of Proposition \ref{prop:f_i_overgp-not_VR}]
The first claim follows from the work of Davis and Olshanskii in \cite{Dav-Ols}: see Lemma \ref{lem:Dav-Olsh}. So, it remains to prove the second claim.

By definition, $G=A \wr B$, where $A={\Z_2}^2$ and $B=\langle b \rangle$ is infinite cyclic. Consider the automorphism $\alpha: A \to A$ given by $\alpha\bigl((\overline{1},\overline{0})\bigr)=(\overline{1},\overline{0})$
and $\alpha\bigl((\overline{0},\overline{1})\bigr)=(\overline{1},\overline{1})$.
Evidently we can extend $\alpha$ to an automorphism of $G$ by defining $\alpha(b)=b$.
It is easy to see that $\alpha$ has order $2$ in $Aut(G)$. Set $\tilde G=G \rtimes \langle \alpha \rangle$ and
$H=\C_{\tilde G}(\alpha) \leqslant \tilde G$. Then $|\tilde G:G|=2$ and $H=\langle (\overline{1},\overline{0}),b, \alpha \rangle \cong (\Z_2 \wr \Z)\times \Z_2$.

Suppose  that $H \vr \tilde G$. Then there exists $N \leqslant \tilde G$ which is normalized by $H$, has trivial intersection with $H$ and satisfies $|\tilde G:HN|<\infty$.
Of course we can replace $N$ with $N \cap G$ to assume that $N \subseteq G$. Clearly $|\tilde G:H|=\infty$, so $N$ must contain at least one non-trivial element $fc$,
where $f \in A^B$ and $c \in B$. If $c \neq 1$ then the commutator of $fc$ with $(\overline{1},\overline{0}) \in A \cap H$ is non-trivial and belongs to $A^B$.
On the other hand, if $c=1$ in $B$ then $f \in N \cap A^B$ must be non-trivial. Thus in any case there must exist a non-trivial function $g \in N \cap  A^B$.

Note that $g \notin \C_{\tilde G}(\alpha)= H$ as $H \cap N$ is trivial, so the commutator $[\alpha,g]$ must be a non-trivial element of $N \cap A^B$ ($[\alpha,g] \in N$
because $N$ is normalized by $\alpha \in H$). However, for every $x \in B$ we have $[\alpha,g](x)=[\alpha, g(x)]$ which can either equal to $(\overline{0},\overline{0})$ or to
$(\overline{1},\overline{0})$ in $A$ (this can be checked directly or by noticing that $\langle A, \alpha\rangle$ is isomorphic to the dihedral group of order $8$,
whose derived subgroup is the $2$-element subgroup generated by $(\overline{1},\overline{0})$). Thus $[\alpha,g]$ is a non-trivial element from
$N \cap\langle (\overline{1},\overline{0})\rangle^B \subseteq N \cap  H$, which contradicts the triviality of $H \cap N$.

Therefore $H$ cannot be a virtual retract of $\tilde G$, so $\tilde G$ does not have \VR{}.
\end{proof}

\section{Virtual free factors of virtually free groups}\label{sec:vf}
The goal of this section is to establish necessary and sufficient criteria for determining whether a finitely generated subgroup of a finitely generated virtually free
groups is a free factor of a finite index subgroup. Our approach uses actions on trees.
More concretely, by a well-known theorem of Karrass, Pietrowski and Solitar \cite{KPS}, a finitely generated group $G$ is virtually free if and only if it
decomposes as the fundamental group of a finite graph of groups with finite vertex and edge groups.
The action of $G$ on the corresponding Bass-Serre tree gives rise to the following statement.

\begin{lemma}[{\cite[Ch. IV, Corollary 1.9]{D-D}}]\label{lem:virt_free_gp_act_on_trees} Every finitely generated virtually free group admits a cocompact action on
a tree  with finite vertex (and edge) stabilizers.
\end{lemma}

The next lemma is an application of M. Hall's result \cite{M-Hall} stating that free groups are LERF.

\begin{lemma}\label{lem:sep_subtree} Suppose that $G$ is a virtually free group acting on a tree $\cT$ with finite vertex stabilizers.
Let $F \leqslant G$ be a finitely generated subgroup and let $\cX \subseteq \cT$ be an $F$-invariant $F$-cocompact subtree of $\cT$.
Then there exists a finite index subgroup
$K \leqslant G$ such that $F \subseteq K$ and for any $g \in K$, $g \notin F$, one has $g\cX \cap \cX = \emptyset$.
\end{lemma}

\begin{proof} Let $G_0 \lhd G$ be a free normal subgroup of finite index in $G$, and let $F_0=F \cap G_0$. Then $G_0$ acts freely on $\cT$
as the vertex stabilizers for the action of $G$ on $\cT$ are finite. Since $|F:F_0|<\infty$,
the induced action of $F_0$ on $\cX$ is still cocompact, so we can choose a finite set $V$ of representatives of orbits of vertices of $\cX$ under this action.

Now, for every pair of distinct vertices $u,v \in V$ such that $v \in G_0 u$, choose an arbitrary element $g_{uv} \in G_0$ such that $v=g_{uv} u$.
Let $S \subseteq G_0$ denote the finite set of these elements and their inverses. Note that $F_0 \cap S =\emptyset$ and $F_0$ is a finitely generated subgroup of the free group $G_0$.
By M. Hall's Theorem \cite[Theorem 5.1]{M-Hall}, there is a finite index normal subgroup $N \lhd G_0$ such that $F_0N \cap S= \emptyset$. Since $|G:N|<\infty$,
after replacing $N$ with its normal core (the intersection of all its conjugates in $G$), we can assume that $N \lhd G$ (and, still, $N \subseteq G_0$, $|G:N|<\infty$).

Observe that $K=F N$ is a finite index subgroup of $G$ containing $F$, and denote $K_0=F_0N \leqslant G_0$.
To check that $K$ satisfies the desired property, let us first suppose that $h u=v$ for some $u,v \in V$,
$u \neq v$, and some $h \in K_0$. Then $v=g u $ for some $g \in S$, hence $hu=gu$, so $h=g \in G_0$ as $G_0$ acts freely on $\cT$. Thus we arrive to a contradiction:
$h\in K_0 \cap S=\emptyset$. Therefore distinct vertices from $V$ belong to different $K_0$-orbits.

Consider any  $g\in K$ such that $g\cX \cap \cX \neq \emptyset$. Then $g=fh$ for some $f \in F$ and $h \in N$,
and there exist vertices $x,y$ of $\cX$ such that $g x =y$;
thus $h x=f^{-1}y$.
Note that $z=f^{-1} y$ is still a vertex of $\cX$, as $\cX$ is $F$-invariant, so there exist vertices $u,v \in V$ and elements $f_1,f_2 \in F_0$ such that
$x =f_1u$ and $z=f_2 v$. The equality $hf_1u=f_2 v$ shows that $u$ and $v$ are in the same $K_0$-orbit, as $hf_1,f_2 \in K_0$. It follows that $u=v$, so $hf_1u=f_2 u$,
which yields that $h f_1=f_2$, as $hf_1,f_2 \in G_0$
and $G_0$ acts freely on $\cT$. Hence $h=f_2f_1^{-1} \in F_0$, and so $g=fh \in F$, as required.
\end{proof}

Recall that a subgroup $F$ of a group $G$ is said to be \emph{malnormal} if for every $g \in G \setminus F$ one has $gFg^{-1} \cap F=\{1\}$.
We will say that $F$ is \emph{v-malnormal} in $G$ if there exists a finite index subgroup $K \leqslant G$, containing $F$, such that $F$ is malnormal in $K$.

Suppose that $G$ acts on a tree $\cT$ and $\cX$ is a subtree of $\cT$. For every subtree $\mathcal Y$ of $\cT$ we will use $\St_F(\mathcal{Y})$ to denote the
\emph{pointwise stabilizer} of $\mathcal{Y}$ in $F$.
We will also use $\partial_{\cT} \cX$ to denote the set of edges of $\cT$ each of which starts at a vertex in $\cX$ and ends at a vertex
in $\cT \setminus \cX$.

The main result of this section is the following expanded version of Theorem \ref{thm:Hall-crit-simple} from the Introduction.

\begin{thm} \label{thm:Hall_for_vf} Let $G$ be a finitely generated virtually free group and let $F \leqslant G$ be a finitely generated subgroup.
Then the following are equivalent:

\begin{itemize}
  \item[(i)] there is a finite index subgroup $K \leqslant G$ such that $F \subseteq K$ and $F$ is a free factor of $K$;
  \item[(ii)] $F$ is v-malnormal in $G$;
  \item[(iii)] $|\C_G(f):\C_F(f)|<\infty$ for each finite order element $f \in F\setminus\{1\}$;
  \item[(iv)] for every cocompact action of $G$ on a tree $\cT$ with finite vertex stabilizers, there exists an $F$-invariant $F$-cocompact subtree $\cX$ of $\cT$
such that $\St_F(e)=\{1\}$ for all edges $e \in \partial_{\cT} \cX$;
\item[(v)] there exist a cocompact action of $G$ on a tree $\cT$ with finite vertex stabilizers and an $F$-invariant $F$-cocompact subtree $\cX$ of $\cT$
such that $\St_F(e)=\{1\}$ for all edges $e \in \partial_{\cT} \cX$.
\end{itemize}
\end{thm}

\begin{proof} Claim (i) implies claim (ii), as in a free product a free factor is always malnormal (this is an easy consequence of the Normal Form Theorem). To show that (ii)
implies (iii), suppose that $|\C_G(f):\C_F(f)|=\infty$ for some non-trivial element $f \in F$, and let $K \leqslant G$ be any subgroup of finite index containing $F$.
Then $|\C_G(f):\C_K(f)|<\infty$, so $|\C_K(f):\C_F(f)|=\infty$. In particular, there must exist an element $g \in \C_K(f) \setminus F$, which yields that
$f \in gFg^{-1}\cap F \neq\{1\}$. Thus, $F$ is not malnormal in $K$. Hence we have proved that (ii) implies (iii) by the contrapositive.

Let us now show that (iii) implies (iv). Suppose that $G$ acts cocompactly on a tree $\cT$ with finite vertex stabilizers.
Note that then the tree $\cT$ must be locally finite: indeed, if a vertex $v$ was incident to infinitely many edges in $\cT$, then infinitely many
of these edges (starting at $v$) would belong to the same $G$-orbit, as $G \backslash \cT$ is finite. The latter would yield that $|\St_G(v)|=\infty$,
contradicting the assumptions.

By Lemma \ref{lem:min_inv_subtree}, $\cT$ has an $F$-invariant subtree $\cX_0$ on which $F$ acts cocompactly.
Let $G_0\lhd G$ be a normal free subgroup of $G$, with $|G:G_0|<\infty$. Then $G_0$ acts freely and cocompactly on $\cT$. Let $N \in \N$ be the number of $G_0$-orbits
of the (oriented) edges in $\cT$, and let $\cX$ be the $N$-neighborhood of $\cX_0$ in $\cT$. In other words, $\cX$ is the subtree of $\cT$ spanned on the
vertices of $\cT$ which are at distance at most $N$ from $\cX_0$. Obviously $\cX$ is still $F$-invariant. Moreover, $F$ acts on $\cX$ cocompactly because
$F\backslash \cX_0$ is finite and $\cT$ is locally finite.

Let $e$ be any edge in $\partial_{\cT} \cX$ and let $f \in \St_F(e)$. Then $f \in F$ must have finite order, so it must fix a vertex of $\cX_0$
(cf. \cite[Ch. I, Corollary 4.9]{D-D}). Let $v$ be the vertex of $\cX_0$ which is fixed by $f$ and is closest to $e_-$.
Then $f$ fixes pointwise the oriented geodesic edge path $e_1 e_2 \dots e_n$, where $(e_1)_-=v$, $(e_i)_+=(e_{i+1})_-$,
$i=1,\dots,n-1$, and $e_n=e$. Note that $n > N$ since $e_+ \notin \cX$ and $\cX$ is the $N$-neighborhood of $\cX_0$. Therefore there must exist indices $i,j$,
$1 \le i<j \le n$, and an element $g \in G_0$ such that $e_j=g\,e_i$. Since $(e_i)_+ \notin \cX_0$ by the choice of $v$,
$e_i$ is not an edge of $\cX_0$. Moreover, the element $g$ translates the edge $e_i$, by construction. It follows that $g$ is hyperbolic and $e_i$ is on the axis of $g$; in particular, for every $k \in \N$, $g^k$ is hyperbolic and translates $e_i$.
The latter implies that $g^k \notin F$ for all $k \in \N$, as otherwise, $\cX_0$ would contain all edges translated by $g^k$ (see Lemma \ref{lem:min_inv_subtree}),
but $e_i$ is not an edge of $\cX_0$. Hence $\langle g \rangle \cap F=\{1\}$.

Observe that the commutator $g^{-1}f^{-1}gf$ belongs to $G_0$ (as $G_0 \lhd G$) and fixes $e_i$ (as $f\,e_i=e_i$, $f\,e_j=e_j$). Since $G_0$ acts on $\cT$ freely, this commutator must be trivial,
hence $g \in \C_G(f)$. It follows that $\C_G(f)$ contains the infinite cyclic subgroup $\langle g \rangle$, which has trivial intersection with $F$. Therefore
$|\C_G(f):\C_F(f)|=\infty$, so $f=1$ by (ii). Thus $\St_F(e)=\{1\}$ and (iv) holds.

The implication (iv) $\Rightarrow$ (v) is clear in view of Lemma \ref{lem:virt_free_gp_act_on_trees}. Thus it remains to prove that (v) implies (i).

Suppose that (v) holds. By Lemma \ref{lem:sep_subtree} we can find a finite index subgroup $K \leqslant G$ such that $F \subseteq K$ and for all $g \in K \setminus F$,
$g\cX$ is disjoint from $\cX$ in $\cT$. It follows that for all $g,h \in K$ either $g\cX=h\cX$ (if $h^{-1}g \in F$) or $g\cX \cap h\cX=\emptyset$ (if $h^{-1}g \notin F$).
Therefore the action of $K$ on $\cT$ induces an action of $K$ on a new tree $\cT'$,
which is obtained from $\cT$ by contracting $\cX$ and all its translates $g\cX$, $g \in K$, to single vertices.

Let $x'$ be the vertex of $\cT'$ obtained from $\cX$. The choice of $K$ implies that $\St_K(x')=F$. On the other hand, the preimage of any edge $e'$ of $\cT'$, starting at
$x'$, is a single edge $e \in \partial_{\cT}(\cX)$ starting at some vertex $v \in \cX$. Therefore, if $g \in \St_K(e')$, then $g \in \St_K(e)$, so
$gv=v \in g\cX \cap \cX$. The latter implies that $g \in F$, so $g \in \St_F(e)=\{1\}$ by (v). Hence
\begin{equation}\label{eq:triv_ed_stab}
\St_K(e')=\{1\} \text{ for all edges $e'$ of $\cT'$, with $e'_-=x'$}.
\end{equation}
Now, by the Structure Theorem for groups acting on trees (see \cite[Chapter I.5.4, Theorem~13]{Serre} or \cite[Chapter I.4, Theorem 4.1]{D-D}),
the action of $K$ on $\cT'$ gives rise to an isomorphism between $K$ and the
fundamental group of the quotient graph of groups. Moreover, in view of \eqref{eq:triv_ed_stab}, the definition of this fundamental group using generators and relators
(see \cite[Chapter I.5.1]{Serre} or \cite[Chapter I.4.1]{D-D}) immediately implies that
$\St_K(x')=F$ is a free factor of $K$. Thus (i) holds, and the proof is complete.
\end{proof}

\begin{rem}\label{rem:suff_edge_elts} If $G$ is given as the fundamental group of a finite graph of finite groups and $f \in F \setminus\{1\}$ is an element of finite order
then $|\C_G(f)|<\infty$ unless $f$ is conjugate to an element from one of the edge groups in $G$. Thus  condition (iii) of Theorem \ref{thm:Hall_for_vf} only needs to be checked for elements conjugate into the edge groups.
\end{rem}

Indeed, since $f$ has finite order, it must fix a vertex $v$ of the Bass-Serre tree $\cT$ for the given splitting of $G$.
If $f$ is not conjugate to any edge group in $G$, then $f$ does not stabilize any edge of $\cT$, so the fixed point set of $f$ is $\{v\}$. It follows that $\C_G(f)$ also fixes $v$, hence $|\C_G(f)|<\infty$.

The following two consequences of Theorem~\ref{thm:Hall_for_vf} have already been known.

\begin{cor}[{cf. \cite[Theorem 1.1]{Burns2}}] \label{cor:Hall_free_prods} Let $G$ be the free product of finitely many finite groups and a finitely generated free group.
Then every finitely generated subgroup $F \leqslant G$ is a free factor of a finite index subgroup. In particular, $G$ satisfies property \VR{}.
\end{cor}

\begin{proof} The given group $G$ is certainly finitely generated and virtually free, and it can be represented as the fundamental group of a finite graph of groups with
finite vertex groups and trivial edge groups. The action on the corresponding  Bass-Serre tree will have finite vertex stabilizers and trivial edge stabilizers.
Thus the corollary now follows  from a combination of Lemma \ref{lem:min_inv_subtree} with the implication (v) $\Rightarrow$ (i) of Theorem~\ref{thm:Hall_for_vf}.
\end{proof}

The implication (iii) $\Rightarrow$ (i)  of Theorem~\ref{thm:Hall_for_vf} yields the following.

\begin{cor}[{cf. \cite[Corollary 3.1]{Bogop-2}}] \label{cor:Hall_for_tors-free} If $G$ is a finitely generated virtually free group and $F \leqslant G$ is a finitely generated torsion-free subgroup, then
$F$ is a free factor of a finite index subgroup. In particular, $F \vr G$.
\end{cor}

Finally, let us prove Corollary \ref{cor:Brun-Burns} from the Introduction.
\begin{proof}[Proof of Corollary \ref{cor:Brun-Burns}] Let $G$ be a finitely generated virtually free group.
If $G$ satisfies M. Hall's property then every non-trivial finitely generated subgroup $F \leqs G$ is
v-malnormal by Theorem~\ref{thm:Hall_for_vf}, hence $|\No_G(F):F|<\infty$.

Conversely, suppose that $|\No_G(H):H|<\infty$ for every non-trivial finite subgroup $H \leqs G$, and let $F \leqs G$ be
any finitely generated subgroup. Then $|\No_G(\langle f \rangle):\langle f \rangle|<\infty$, so the normalizer $\No_G(f)$ must be finite, for all finite order elements $f \in F\setminus\{1\}$.
It follows that
$|\C_G(f)|<\infty$, hence $|\C_G(f):\C_F(f)|<\infty$. Therefore $F$ is a free factor of a finite index subgroup of $G$ by  Theorem~\ref{thm:Hall_for_vf}.
\end{proof}

\section{Open problems}\label{sec:open_q}
In this section we list some open problems motivated by our discussion of virtual retraction properties.

\begin{question}\label{q:v_free->VR} Do virtually free groups satisfy \VR{}?
\end{question}

If the answer to the previous question is positive the next natural step would be to ask the following.

\begin{question}\label{q:v_c_spec->VR} Suppose that $G$ is a virtually compact special hyperbolic group (in the sense of Haglund and Wise \cite{H-W}).
Is every quasiconvex subgroup a virtual retract of $G$?
\end{question}

It follows from the results of \cite{H-W} that every quasiconvex subgroup of a virtually compact special hyperbolic group $G$ contains a finite index subgroup which is a virtual retract of $G$. Thus both Questions \ref{q:v_free->VR} and \ref{q:v_c_spec->VR} are particular instances of
Question \ref{q:quest-1} from the Introduction.

Following the discussion in Remark \ref{rem:claim_in_L-R} we also need to ask the following.

\begin{question}\label{q:L-R} Suppose that $G$ is a finitely generated linear group (over $\mathbb{C}$) and $K \leqslant G$
is a finite index subgroup satisfying \VR{}. Does $G$ also satisfy \VR{}?
\end{question}

Recall that a group $G$ is said to be \emph{large} if some finite index subgroup $K \leqs G$ admits an epimorphism onto a non-abelian free group.
\begin{question}\label{q:f_p_small_with_VRC} Does there exist a finitely presented torsion-free group with \VRC{} which is neither virtually abelian nor large?
\end{question}

In view of Proposition \ref{prop:VRC->inf_vbn}, a positive answer to Question \ref{q:f_p_small_with_VRC} would give the first example
of a finitely presented group which has infinite first virtual betti number but is not large (see \cite[Introduction]{Brid-Koch}).
Theorem \ref{thm:wreath->VRAb} provides examples of non-large groups with \VRC{}, but such wreath products are almost never finitely presented
(cf. \cite[Theorem 1]{Baumslag}).

Corollary \ref{cor:amalg_over_virt_cyclic} naturally suggests the following.

\begin{question}\label{q:VRC_for_amalg} Is property \VRC{} stable under taking amalgamated free products over (virtually) cyclic subgroups?
\end{question}

Using quasi-potency and Lemma \ref{lem:retr_onto_virt_ab->homom} this question can be reduced to asking whether the amalgamated product of two finitely generated virtually
abelian groups over a virtually cyclic subgroup satisfies \VRC{}. The latter is related to the following interesting problem.

\begin{question}\label{q:graph_of_ab->VRC} Let $G$ be the fundamental group of a finite graph of groups with free abelian vertex groups (of finite ranks) and cyclic edge groups. If $G$ is cyclic subgroup separable, does it necessarily satisfy \VRC{}?
\end{question}

\section{Appendix: one example}\label{sec:appendix}
In \cite{Bogop-1} Bogopol'ski\v{\i} used the covering theory for finite $3$-complexes to obtain a characterization for the fundamental group $G$, of a finite graph of finite groups, to satisfy M. Hall's property. And in \cite{Bogop-2} he developed an algorithm for checking whether a finitely generated subgroup is a free factor of a subgroup of finite index in $G$.
Unfortunately, \S 11 of \cite{Bogop-1} claimed to give a counter-example to the conjecture of Brunner and Burns \cite{Brun-Burns} discussed in the Introduction.
This would contradict our Corollary~\ref{cor:Brun-Burns}, hence the goal this appendix is to explain that the example from \cite[\S 11]{Bogop-1} is not actually valid. After seeing an earlier draft of this paper, Bogopol'ski\v{\i} confirmed that there is a mistake in this example, but the main results from \cite{Bogop-1,Bogop-2} are correct.

More precisely, \cite[\S 11]{Bogop-1} considers the group $G$ which is the fundamental graph of the graph of groups $\mathcal{G}$, with $3$ vertices and $3$ edges, sketched
 on Figure~\ref{fig:Bogop-ex}.
\begin{figure}[!ht]
  \begin{center}
   \includegraphics{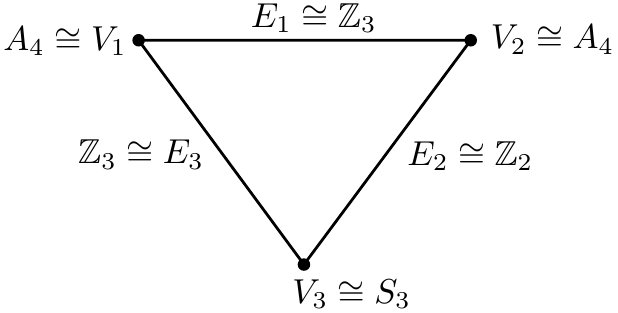}
  \end{center}
\caption{The graph of groups $\mathcal{G}$.}\label{fig:Bogop-ex}
\end{figure}

The vertex groups $V_1$ and $V_2$ are both isomorphic to the alternating group $A_4$ and the vertex group $V_3$ is isomorphic to the symmetric group $S_3$. The edge groups $E_1$ and $E_3$ are cyclic of order $3$,
the edge group $E_2$ is cyclic of order $2$, and the embeddings of $E_1$ and $E_3$ in $V_1$ have the same image. Let us take $\langle a_1,a_2 \,\|\, a_1^3,a_2^3,(a_1a_2)^2\rangle$
as a presentation for $V_1$, $\langle b_1,b_2 \,\|\, b_1^3,b_2^3,(b_1b_2)^2\rangle$
as a presentation for $V_2$  and $\langle c_1,c_2 \,\|\, c_1^3,c_2^2,(c_1c_2)^2\rangle$
as a presentation for $V_3$. In \cite[\S 11]{Bogop-1} it is stated that the particular embeddings of the edge groups of $\mathcal G$
into the vertex groups are arbitrary, therefore, after choosing the maximal tree in $\mathcal{G}$ to consist of the edges corresponding to $E_3$ and $E_2$,
we can assume that $G=\pi_1(\mathcal{G})$ has the following presentation:

\[G\cong\langle a_1,a_2,b_1,b_2,c_1,c_2,t \,\|\, a_1^3,a_2^3,(a_1a_2)^2,b_1^3,b_2^3,(b_1b_2)^2,c_1^3,c_2^2,(c_1c_2)^2, ta_1t^{-1}b_1^{-1}, a_1c_1^{-1},b_1b_2c_2^{-1} \rangle.
\]

Clearly we can eliminate $c_1$ and $c_2$ from this presentation to simplify it:
\begin{equation}\label{eq:simp_pres}
G\cong\langle a_1,a_2,b_1,b_2,t \,\|\, a_1^3,a_2^3,(a_1a_2)^2,b_1^3,b_2^3,(b_1b_2)^2,(a_1b_1b_2)^2, ta_1t^{-1}b_1^{-1} \rangle.
\end{equation}
 It is not difficult to see (in view of Remark \ref{rem:suff_edge_elts})
that for any non-trivial element $g$ in $V_1=\langle a_1,a_2\rangle$ or in $V_2=\langle b_1,b_2 \rangle$, the centralizer $\C_G(g)$ is finite.
However Proposition 11.1 of \cite{Bogop-1} claims that neither of the subgroups $V_1$, $V_2$ is a free factor of a finite index subgroup of $G$, contradicting our Theorem \ref{thm:Hall_for_vf}.
Below we demonstrate, using \cite{GAP4}, that the conclusion of Theorem \ref{thm:Hall_for_vf} is indeed correct in this case
(the GAP code is available from the author upon request).

Observe that $G$ admits a homomorphism $\psi$ onto the alternating group $A_5$, defined as follows (note that here we multiply permutations from left to right, following the convention used in
GAP):
\begin{equation*} a_1 \stackrel{\psi}{\mapsto} (1\,2\,3), ~a_2 \stackrel{\psi}{\mapsto} (2\,3\,4), ~b_1 \stackrel{\psi}{\mapsto} (2\,4\,3),~
  b_2 \stackrel{\psi}{\mapsto} (3\,5\,4), ~t \stackrel{\psi}{\mapsto} (1\,4\,2).
\end{equation*}

It is easy to check that $\psi$ is injective on the subgroups $V_1$, $V_2$ and $V_3=\langle c_1,c_2\rangle=\langle a_1,b_1b_2\rangle$. Thus $\psi(V_1) \cong A_4 \cong \psi(V_2)$, so
the full preimage $G_0=\psi^{-1}(\psi(V_2))=\psi^{-1}(\langle (2\,4\,3),(3\,5\,4)\rangle )$ is a subgroup of index $5$ in $G$.
Evidently, $V_2 \subseteq G_0$. Moreover, since $\psi(V_2)=\St_{A_5}(1)$ and $\psi(V_1)=\St_{A_5}(5)$ are conjugate in
$A_5=\psi(G)$, a conjugate of $V_1$ in $G$ will also be contained in $G_0$. We used \cite{GAP4} to find the following presentation of $G_0$:
\begin{equation}\label{eq:pres_of_G_0} G_0 \cong \langle f_1,f_2,f_3,f_4,f_5,f_6,f_7 \,\|\,  f_1^3, f_2^3, f_3^3, (f_1f_2)^2, (f_6f_1f_6^{-1}f_3^{-1})^2 \rangle.
\end{equation}
After rewriting the last relator as $(f_1f_6^{-1}f_3^{-1}f_6)^2$ and replacing the generator $f_3$ with $\tilde f_3=f_6^{-1}f_3^{-1}f_6$, we obtain
\begin{equation}\label{eq:short_pres_of_G_0} G_0 \cong \langle f_1,f_2,\tilde f_3,f_4,f_5,f_6,f_7 \,\|\,  f_1^3, f_2^3, \tilde f_3^3, (f_1f_2)^2, (f_1\tilde f_3)^2 \rangle.
\end{equation}
The latter shows that
\[G_0 \cong \bigl(\langle f_1,f_2\,\|\, f_1^3, f_2^3,  (f_1f_2)^2 \rangle \ast_{f_1=\tilde f_1} \langle \tilde f_1,\tilde f_3\,\|\, \tilde f_1^3, \tilde f_3^3,  (\tilde f_1\tilde f_3)^2\rangle \bigr)
*\langle f_4,f_5,f_6,f_7\,\|\, \rangle. \] Therefore
\begin{equation}\label{eq:decomp_of_G_0}
G_0\cong (A_4*_{\Z_3} A_4)*F_4,
\end{equation}
where $F_4$ denotes the free group of rank $4$. Since $V_2$ and a conjugate of $V_1$ are contained in $G_0$, each of these subgroups will be conjugate (in $G$) to one of the $A_4$-factors
in the splitting \eqref{eq:decomp_of_G_0} of $G_0$. In view of the symmetry of this splitting,
to show that $V_1$ and $V_2$ are free factors of finite index subgroups of $G$, it is sufficient to prove that this is the case for
the subgroup $W=\langle f_1,f_2 \rangle \cong A_4$ in $G_0$, where $G_0$ is equipped with the presentation \eqref{eq:short_pres_of_G_0}.

Consider the epimorphism $\varphi:G_0 \to A_5$ defined by
\begin{equation*} f_1 \stackrel{\varphi}{\mapsto} (2\,3\,4), ~f_2 \stackrel{\varphi}{\mapsto} (1\,2\,3),~
  \tilde f_3 \stackrel{\varphi}{\mapsto} (3\,4\,5), \mbox{ and } f_i \stackrel{\varphi}{\mapsto} id~ \mbox{ for } i=4,5,6,7.
\end{equation*}

Once again the full preimage $G_1=\varphi^{-1}(\varphi(W))=\varphi^{-1}(\langle(1\,2\,3),(2\,3\,4)\rangle)$ has index $5$ in $G_0$, and contains $W$. \cite{GAP4} computes the following presentation
of $G_1$ with $24$ generators and $6$ relators:
\begin{equation}\label{eq:pres_of_G_1} G_1 \cong \langle h_1,\dots,h_{24} \,\|\,  h_1^3, h_2^3,(h_2h_1)^2,  h_7^3, h_{20}^3, (h_{20}h_7^{-1})^2 \rangle.
\end{equation}
After replacing $h_7$ with its inverse we see that
\[G_1 \cong A_4*A_4*F_{20}.\]
Obviously, the subgroup $W$ must be conjugate to one of the $A_4$-factors in $G_1$, hence $W$ is itself a free factor of $G_1$, and $|G:G_1|=|G:G_0|\,|G_0:G_1|=25$.
Thus we have checked that for each $i=1,2$, $V_i$ is a free factor of a subgroup of finite index in $G$. The latter confirms the conclusion of our Theorem~\ref{thm:Hall_for_vf},
and shows that the example from  \cite[\S 11]{Bogop-1} is not a counter-example to the conjecture of Brunner and Burns.

\end{document}